\documentclass[a4paper,12pt]{article}

\usepackage{amsmath,amsfonts,amssymb}
\usepackage{amsthm}
\usepackage{abstract}
\usepackage{color,multicol}
\usepackage{appendix}
\usepackage{enumerate}
\usepackage[shortlabels]{enumitem}

\newtheorem{lemma}{Lemma}[section]
\newtheorem{theorem}[lemma]{Theorem}

\newtheorem{remark}[lemma]{Remark}

\usepackage{tocloft}

\begin{document}
\begin{titlepage}
\begin{center}
\vspace*{1in}
{\Large \textbf{Shamir-Duduchava factorization of elliptic symbols}}
\vspace{10mm}
\par
{\textit{Dedicated to Roland Duduchava on the occasion of his 70th birthday.}}
\par
\vspace{1.5in}
{Tony Hill}
\par
\vfill
\par
\vspace{0.5in}
Department of Natural and Mathematical Sciences
\par
\vspace{0in}
King's College London
\par
\vspace{0.5in}
May 2015
\end{center}
\end{titlepage}
\renewcommand{\abstractnamefont}{\normalfont\large\bfseries}
\renewcommand{\abstracttextfont}{\normalfont\normalsize}
\begin{abstract}
This paper considers the factorization of elliptic symbols which can be represented by matrix-valued functions. Our starting point is a \textit{Fundamental Factorization Theorem},  due to Budjanu and Gohberg \cite{BG1}. We critically examine the work of Shamir \cite{Sh}, together with some corrections and improvements as proposed by Duduchava \cite{Du}.  As an integral part of this work, we give a new and detailed proof that certain sub-algebras of the Wiener algebra on the real line satisfy a sufficient condition for right standard factorization. Moreover, assuming only the Fundamental Factorization Theorem, we provide a complete proof of an important result from Shargorodsky \cite{Shar}, on the factorization of an elliptic homogeneous matrix-valued function, useful in the context of the inversion of elliptic systems of multidimensional singular integral operators in a half-space.
\end{abstract}
\newpage
%
%
%
\newpage
\tableofcontents
\newpage
\section{Introduction}
This paper considers the factorization of elliptic symbols which can be represented by matrix-valued functions. Our starting point is a \textit{Fundamental Factorization Theorem},  due to Budjanu and Gohberg \cite{BG1}. We critically examine the work of Shamir \cite{Sh}, together with some corrections and improvements as proposed by Duduchava \cite{Du}. We shall call the combined efforts of these two latter authors the \textit{Shamir-Duduchava} factorization method.  \\

One important application of the Shamir-Duduchava factorization method has been given by Shargorodsky \cite{Shar}. Our primary goal is to provide, in a single place, a complete proof of Shargorodsky's result on the factorization of a matrix-valued elliptic symbol, assuming only the Fundamental Factorization Theorem. As an integral part of this work, we will give a new and detailed proof that certain sub-algebras of the Wiener algebra on the real line satisfy a sufficient condition for right standard factorization.  

\section{Background}
Let $\Gamma$ denote a simple closed smooth contour dividing the complex plane into two regions $D_+$ and $D_-$, where for a bounded contour we identify $D_+$ with the domain contained within $\Gamma$. We shall be especially interested in the case where $\Gamma = \dot{\mathbb{R}}$, the one point compactification of the real line.  In this situation, of course, $D_\pm$ are simply the upper and lower half-planes respectively. We let $G_\pm$ denote the union $D_\pm \cup \Gamma$. 

\subsection{Factorization}
Suppose we are given a nonsingular matrix-valued function $A(\zeta) = \big ( a_{jk}(\zeta) )^N_{j,k=1}$, then we define a $\textit{right standard factorization}$ or simply the $\textit{factorization}$ as a representation of the form
\begin{equation} \label{eq:rightfact}
A(\zeta) = A_-(\zeta) D(\zeta) A_+(\zeta) \qquad (\zeta \in \Gamma),
\end{equation}
where $D(\zeta)$ is strictly diagonal with non-zero elements $d_{jj} = ( (\zeta -\lambda^+) / (\zeta -\lambda^-) )^{ \kappa_j }$ for $j=1, \dots, N$. The exponents $\kappa_1 \geq \kappa_2 \geq \dots \geq \kappa_N$ are integers and $\lambda^\pm$ are certain fixed points chosen in $D_\pm$ respectively. (In passing, we note that if $\Gamma = \dot{\mathbb{R}}$, it is customary to take $\lambda^\pm = \pm i$.) $A_\pm(\zeta)$ are square $N \times N$ matrices  that are analytic in $D_\pm$ and continuous in $G_\pm$. Moreover,  the determinant of $A_+ (A_-)$ is nonzero on $G_+ (G_-)$. \\

As one would expect, interchanging the matrices $A_-(\zeta)$ and $A_+(\zeta)$ in $\eqref{eq:rightfact}$ gives rise to a $\textit{left standard factorization}$. In either a right or a left factorization, the integers $\kappa_j = \kappa_j (A)$ are uniquely determined (see \cite {GK}) by the matrix $A(\zeta)$. Further, if the matrix $A(\zeta)$ admits a factorization for a pair of points $\lambda^\pm$, then it admits the a factorization of the same type for any pair of points $\mu^\pm \in D_\pm$, in that the $\textit{right}$ or $\textit{left indices}$, denoted by $\{ \kappa_j(A), j=1,\dots, N\}$, are independent of the points $\lambda^\pm$. 

\subsection{Banach algebras of continuous functions} \label{BActsfuns}
We let $\mathcal{U}(\Gamma)$ denote a Banach algebra of continuous functions on $\Gamma$ which includes the set of all rational functions $R(\Gamma)$ not having any poles on $\Gamma$. Further we insist that $\mathcal{U}(\Gamma)$ is $\textit{inverse closed}$ in the sense that if $a(\zeta) \in \mathcal{U}(\Gamma)$ and $a(\zeta)$ does not vanish anywhere on $\Gamma$, then $a^{-1}(\zeta) \in \Gamma$. Of course, $\mathcal{U}(\Gamma) \subset C(\Gamma)$, where $C(\Gamma)$ is the Banach algebra of all continuous functions on $\Gamma$, with the usual supremum norm. \\

Consider the region $G_+$. We let $R^+(\Gamma)$ denote the set of all rational functions not having any poles in this domain and $C^+(\Gamma)$ denote the closure of $R^+(\Gamma)$ in $C(\Gamma)$ with respect to the norm of $C(\Gamma)$. It is easy to see that  $C^+(\Gamma)$ is a subalgebra of $C(\Gamma)$ consisting of those functions that have analytic continuation to $D_+$ and which are continuous on $G_+$. We can now define $\mathcal{U}^+(\Gamma) = \mathcal{U}(\Gamma) \cap C^+(\Gamma)$. Again, it is straightforward to show that $\mathcal{U}^+(\Gamma)$ is a subalgebra of $\mathcal{U}(\Gamma)$. (Similar definitions of $C^-(\Gamma)$ and $\mathcal{U}^-(\Gamma)$ follow by considering the region $G_-$.)

\subsection{Splitting algebras}
It turns out that the ability to factorize a given matrix is intimately linked to the ability to express $\mathcal{U}(\Gamma)$ as a direct sum of two subalgebras - one containing analytic functions defined on $D_+$ and the other analytic functions on $D_-$.  To ensure uniqueness of this partition we let ${\mathring{\mathcal{U}}}^- (\Gamma)$ denote the subalgebra of $\mathcal{U}^-(\Gamma)$ consisting of all functions that vanish at the chosen point $\lambda^- \in D_-$. We now say that a Banach algebra $\mathcal{U}(\Gamma) \, \textit{splits}$ if we can write
\begin{equation*} 
\mathcal{U}(\Gamma) = \mathcal{U}^+(\Gamma) \oplus {\mathring{\mathcal{U}}}^-(\Gamma).
\end{equation*}

The prototypical example of a splitting algebra is the Wiener algebra, $W({\mathbb{T}})$, of all functions defined on $\mathbb{T}$, the unit circle $|\zeta| = 1$, of the form
\begin{equation*} 
a(\zeta) = \displaystyle \sum^\infty_{j = -\infty} a_j \zeta^j \quad \bigg (\sum^\infty_{j = -\infty} |a_j| < \infty \bigg )
\end{equation*}
with the norm $\| a(\zeta) \| = \sum^\infty_{j = -\infty} |a_j|$. The Banach algebras $W^\pm({\mathbb{T}})$ have a simple characterisation. For example, $W^+({\mathbb{T}})$ consists of all functions in $W(\mathbb{T})$ that can be expanded as an absolutely converging series in nonnegative powers of $\zeta$. However, the algebra $C({\mathbb{T}})$ does not split. For more details see \cite{BG1}. 

\subsection{R-algebras}
We say that a Banach algebra $\mathcal{U}(\Gamma)$ of complex-valued functions continuous on $\Gamma$ is an \textit{R-algebra} if the set of all rational functions $R(\Gamma)$ with poles not lying on $\Gamma$ is contained in $\mathcal{U}(\Gamma)$ and this set is dense, with respect to the norm of $\mathcal{U}(\Gamma)$. In passing, we note that any R-algebra of continuous functions is inverse closed. (See, for example Chapter 2, Section 3, p.\,44  \cite{CG}.) Following Theorem 5.1, p.\,20  \cite{BG2}, we have:
\begin{theorem} \label{thm:ralg}
\textbf{(Fundamental Factorization Theorem.)} Let $\mathcal{U}(\Gamma)$ be an arbitrary splitting $R$-algebra. Then every nonsingular matrix-valued function $A(\zeta) \in \mathcal{U}_{N \times N} (\Gamma)$ admits a right standard factorization with factors $A_\pm(\zeta)$ in the subalgebras $\mathcal{U}^\pm_{N \times N}(\Gamma)$.
\end{theorem}

\subsection{Wiener algebras on the real line}
Let $L_1(\mathbb{R})$ denote the usual convolution algebra of Lebesgue integrable functions on the real line. For any $g \in L_1(\mathbb{R})$ we define the \textit{Fourier Transform} of $g$ as the function $\mathcal{F} g$, or $\widehat{g}$, given by
\begin{equation*}
(\mathcal{F} g)(t) = \widehat{g}(t) := \dfrac{1}{\sqrt{2 \pi}} \int^\infty_{-\infty} g(x) e^{ixt} dx.
\end{equation*} 

We let $C^\infty_0(\mathbb{R})$ denote the algebra of continuous functions $f$ on $\mathbb{R}$ which vanish at $\pm \infty$. It is well known, (see, for example, Chapter 9, Theorem 9.6, p.\,182 \cite{Ru}), that if $g \in L_1(\mathbb{R})$, then \begin{equation} \label{eq:hatnorm}
\widehat{g} \in C^\infty_0(\mathbb{R}); \quad \| \widehat{g} \|_\infty \leq \| g \|_1. \\
\end{equation}


The \textit{Wiener algebra} $W(\mathbb{R})$ is the set of all functions of the form $f= \widehat{g} + c$ where $g \in L_1(\mathbb{R})$ and $c$ is a constant. The norm on $W(\mathbb{R})$ is given by
\begin{equation*}
\| f \|_{W(\mathbb{R})} = \|g\|_1+ |c|
\end{equation*}
Supose $f_1 = \widehat{g}_1 + c_1, f_2 = \widehat{g}_2 + c_2 \in W(\mathbb{R})$. Then since $\widehat{g}_1\widehat{g}_2 = \widehat{g_1 * g_2}$ (see, for example,  Chapter 9, Theorem 9.2, p.\,179 \cite{Ru}), it is straightforward to show that $W(\mathbb{R})$ is a Banach algebra. \\

We will also consider certain subalgebras of the Wiener algebra $W(\mathbb{R})$. For $r=0,1,2 ,\dots$ we define $W^r(\mathbb{R})$ to be the set of functions $f$ such that
\begin{equation*}
(1-it)^k D^k f(t) \in W(\mathbb{R}) \quad (k = 0, 1, \dots, r),
\end{equation*}
where $D^k$ is the $k$th order derivative. (Of course, $W^0(\mathbb{R})$ is simply $W(\mathbb{R})$.) We shall show that $W^r(\mathbb{R})$ is a Banach algebra and, moreover, is a splitting $R$-algebra. \\

\subsection{Homogeneity, differentiability and ellipticity}
Suppose $\xi = (\xi_1, \dots \xi_n) \in \mathbb{R}^n$ for some integer $n \geq 2$. It will be convenient to write $\xi = (\xi^\prime, \xi_n)$ where $\xi^\prime \in \mathbb{R}^{n-1}$. We assume that $\mathbb{R}^n$ has the usual Euclidean norm, and we let $\mathbb{S}^{n-1}$ denote the set $\{  \xi \in \mathbb{R}^n \, | \, \xi_1^2 + \dots + \xi_n^2 = 1 \}$. \\

We further suppose that $A_0(\xi^\prime, \xi_n)$ is an $N \times N$ matrix-valued function defined on $\mathbb{R}^n$, which is homogeneous of degree $0$. In addition, we will assume the elements of the matrix $A_0(\xi^\prime, \xi_n)$ belong to  $C^{r+2} (\mathbb{S}^{n-1})$, for some non-negative integer $r$, where $C^r (\mathbb{S}^{n-1})$ denotes the set of $r$ times continuously differentiable functions on the domain $\mathbb{S}^{n-1}$. Finally, we assume that $A_0(\xi^\prime, \xi_n)$ is elliptic, in that
\begin{equation*}
\inf_{\xi \in \mathbb{S}^{n-1}} | \operatorname{det} A_0(\xi)| > 0. \\
\end{equation*}

\subsection{The matrices $E_{\pm}$ and $E$}
We will be particularly  interested in the behaviour of $A_0(\xi^\prime, \xi_n)$ as $\xi_n \rightarrow \pm \infty$. \\

Our approach is effectively to fix $\xi^\prime$, and thereby consider factorization in the one-dimensional (scalar) variable $\xi_n$.  Since $A_0(\xi^\prime, \xi_n)$ is homogeneous of degree zero, 
\begin{equation*}
\lim_{\xi_n \rightarrow \pm \infty} A_0(\xi^\prime, \xi_n) =  A_0(0, \dots, 0, \pm 1), 
\end{equation*}
for fixed $\xi^\prime$. We define
\begin{equation} \label{Epmdefinition}
E_\pm := A_0(0, \dots, 0, \pm 1) \quad \text{and} \quad E:= E^{-1}_+ \, E_-.
\end{equation}

\subsection{The matrices $B_{\pm}$}
It is a standard result that any $E \in \mathbb{C}_{N \times N}$ can be expressed in \textit{Jordan Canonical Form}
\begin{equation*}
h_1 E h^{-1}_1 = J := \text{diag }[J_1, \dots, J_l],
\end{equation*}
where the \textit{Jordan block}  $J_k = J_k (\lambda_k)$ is a matrix of order $m_k$ with eigenvalue $\lambda_k$ on every diagonal entry, $1$ on the super-diagonal and $0$ elsewhere. The matrix $h_1$ is invertible and
\begin{equation*}
m_1 + \dots + m_l = N.
\end{equation*}
The Jordan matrix, $J$, is unique up to the ordering of the blocks $J_k, \, k=1, \dots, l$.  \\

Let $B^m(z)$ to be the $m \times m$ matrix $(b_{jk}(z))^m_{j,k=1}$ given by
\[
 b_{jk}(z) :=
   \begin{cases}
     0  & j < k\\
     1  & j = k\\
     z^{j-k}/(j-k)! & j > k.
  \end{cases}
\]
We now define
\begin{equation} \label{defnB1}
K := \text{diag } \big [ K_1, \dots, K_l \big ],
\end{equation}
where $K_k := \lambda_kB^{m_k}(1)$.  By construction, $K$ is a lower triangular matrix whose block structure \textit{and} diagonal elements are identical to those of $J$. \\

A routine inspection of the equation
\begin{equation*}
K_k u = \lambda_k u
\end{equation*}
shows that the eigenspace associated with the eigenvalue $\lambda_k$ has dimension one. Therefore, see p.\,191 \cite{Dett}, the matrix $K_k$ is similar to the Jordan block $J_k(\lambda_k)$ for $k=1, \dots, l$. Thus $K$ is similar to $J$, and we have 
\begin{equation*}
J = h_2 \, K \, h_2^{-1},
\end{equation*}
for some nonsingular matrix $h_2$. Hence we can write
\begin{equation} \label{EB1}
E = h \, K \, h^{-1}, \quad \text{where} \quad h:= h^{-1}_1 h_2. \\
\end{equation} 

For any $z_1, z_2 \in \mathbb{C}$ and positive integer $m$, it is easy to show that the matrix-valued functions $B^m(z)$ satisfy
\begin{equation} \label{Bab}
B^m(z_1 + z_2) = B^m(z_1 ) B^m(z_2); \quad B^m(0) = I.
\end{equation}
In particular, taking $z_2 = -z_1$ gives
\begin{equation} \label{B-a} 
B^m(- z_1) = [B^m(z_1)]^{-1}. \\
\end{equation}

In the analysis that follows we will be using the logarithm function on the complex plane. Unless specifically stated to the contrary, we will always take the principal branch of the logarithm, Log $z$,  defined by 
\begin{equation*}
\text{Log } z = \log |z| + i \arg z, \quad -\pi < \arg z \leq \pi,
\end{equation*}
for any non-zero $z \in \mathbb{C}$. In other words, we assume that the discontinuity in $\arg z$ occurs across the negative real axis. \\

For any $t \in \mathbb{R}$, we now define the complex-valued functions
\begin{equation} \label{alphapm}
\alpha_\pm(t) := (2 \pi i)^{-1} \log(t \pm i).
\end{equation}

Then
\begin{equation} 
\lim_{t \rightarrow +\infty} [ \alpha_+(t) - \alpha_-(t)] = 0; \quad \lim_{t \rightarrow -\infty} [ \alpha_+(t) - \alpha_-(t)]  = 1.
\end{equation}

Corresponding to the block decomposition in \eqref{defnB1}, we set
\begin{equation} \label{Bpmdef}
B_\pm(t) = \text{diag }\big [ B^{m_1}((2 \pi i)^{-1} \log(t \pm i)), \dots, B^{m_l}((2 \pi i)^{-1} \log(t \pm i)) \big ]. \\
\end{equation}
We note, in passing, that in the special case that $l = N$, then $B_\pm(t) = I$.\\

Following \cite{Sh}, we now give a simple test for membership of $W^r(\mathbb{R})$ for continuously differentiable functions. \\

\begin{lemma} \label{Wrtest}
Let $r=0,1,2, \dots $ and suppose the function $b(t) \in C^{r+1} (\mathbb{R})$ has the property that, for some $\delta > 0$,
\begin{equation*}
D^k b(t) = O ( |t|^{-k-\delta}), \quad k = 0,1, \dots, (r+1)
\end{equation*}
then $b(t) \in W^r(\mathbb{R})$.
\end{lemma}
\begin{proof}
We follow the approach given in \cite{Sh}. For $0 \leq k \leq r$, we define
\begin{equation*}
g_k(t) = (1-it)^k b^{(k)}(t).
\end{equation*}
Our goal is to show that $g_k(t) \in W(\mathbb{R})$. \\

Differentiating with respect to $t$, 
\begin{equation*}
g^\prime_k(t) = -ik(1-it)^{k-1} b^{(k)}(t) + (1-it)^k b^{(k+1)}(t)
\end{equation*}
Then, by hypothesis, $g_k$ and $g^\prime_k$ are continuous. Moreover as $|t| \rightarrow \infty$, we have $g_k(t) = O (|t|^{-\delta})$ and $g^\prime_k(t) = O (|t|^{-1-\delta})$. Hence $g^\prime_k(t) \in  L^2(\mathbb{R})$. \\

On applying the Fourier transform ($\mathcal{F}_{t \rightarrow \xi}$) to the function $g^\prime_k(t)$, we obtain $\xi \widehat{g}_k(\xi) \in L^2(\mathbb{R})$. But, using the Cauchy-Schwarz inequality, 

\begin{equation*}
\int_{| \xi | \geq \epsilon} |\widehat{g}_k(\xi)| \, d\xi = \int_{| \xi | \geq \epsilon} \dfrac{1}{| \xi|} | \xi \widehat{g}_k(\xi)| \, d\xi \leq \bigg ( \int_{| \xi | \geq \epsilon} \dfrac{1}{| \xi|^2} \, d\xi\bigg )^{\frac{1}{2}} \|\xi \widehat{g_k} \|_{L^2} < \infty.
\end{equation*}


Hence, $\widehat{g}_k(\xi)$ is absolutely integrable everywhere outside a neighbourhood $(-\epsilon, \epsilon)$ of zero. On the other hand, for small $|\xi|$, from  Theorem 127, p.\,173 \cite{Ti}, $\widehat{g}_k(\xi) = O( |\xi|^{\delta - 1})$ and hence $\widehat{g}_k(\xi)$ is absolutely integrable inside $(-\epsilon, \epsilon)$. \\

Thus,  $\widehat{g}_k(\xi) \in L^1(\mathbb{R})$. We now define a new function $h_k(x) = \widehat{g}_k(-x)$. Then, by construction, $h_k(x) \in L^1(\mathbb{R})$ and taking the Fourier transform ($\mathcal{F}_{x \rightarrow t}$) of $h_k(x)$ we obtain
\begin{align*}
\widehat{h}_k(t) &= \dfrac{1}{\sqrt{2 \pi}} \int^\infty_{-\infty} \widehat{g}_k(-x) e^{ixt} dx \\
       &= \dfrac{1}{\sqrt{2 \pi}} \int^\infty_{-\infty} \widehat{g}_k(x) e^{-ixt} dx \\
       & = g_k(t)			
\end{align*}
Now $\widehat{h}_k(t) \in W(\mathbb{R})$, and hence, $g_k(t) \in W(\mathbb{R})$. This completes the proof of the lemma.
\end{proof}

\newpage
\subsection{Key theorem from Shamir}
The next theorem, see Appendix pp.\,122-123 \cite{Sh}, considers some properties of a certain matrix-valued function derived from an elliptic homogeneous matrix-valued function of degree zero. Together with Theorem \ref{thm:ralg}, it will provide the starting point for proving our second result.
\begin{theorem} \label{MainResultCase3}
Suppose that $A_0(\xi^\prime, \xi_n) \in C_{N \times N}^{r+3}(\mathbb{S}^{n-1})$ is a matrix-valued function which is homogeneous of degree $0$ and elliptic. Suppose that the Jordan form of $A_0^{-1}(0, \dots, 0,1) A_0(0, \dots, 0, -1)$ has blocks $J_k(\lambda_k)$ of size $m_k$ for $k=1, \dots, l$. Let the matrix $c := A^{-1}_0(0, \dots, 0, 1)$, and the constant invertible matrix $h$ be as in equation \eqref{EB1}.  Then, for fixed $\xi' \not =0$, 
\begin{align*}
\lim_{\xi_n \to +\infty} h^{-1} c A_0(\xi', \xi_n) h & = I; \\
\lim_{\xi_n \to -\infty} h^{-1} c A_0(\xi', \xi_n) h & = \operatorname{diag} \,[\lambda_1B^{m_1}(1), \dots, \lambda_lB^{m_l}(1)]. 
\end{align*}

Further let $\zeta = (\zeta_1, \dots \zeta_N)$, where 
\begin{equation} \label{zetaqdefn}
\zeta_q = - (\log \lambda_j)/(2 \pi i)  \, \text{ for } \sum^{j-1}_{k=1} m_k < q \leq \sum^{j}_{k=1} m_k, \quad q=1, \dots, N,
\end{equation}
and define $(\xi_n \pm i)^\zeta := \operatorname{diag} \, [(\xi_n \pm i)^{\zeta_1}, \dots, (\xi_n \pm i)^{\zeta_N}]$. \\

Then, for fixed $\xi^\prime \not =0$,
\begin{equation*}
A^*_0(\xi', \xi_n) := (\xi_n - i)^{-\zeta} B_-(\xi_n) h^{-1} c A_0(\xi^\prime, \xi_n) h B_+^{-1}(\xi_n) (\xi_n + i)^{\zeta} \in W^{r+2}_{N \times N}(\mathbb {R}),
\end{equation*}
and
\begin{equation} \label{A*0omegat}
\lim_{\xi_n \to \pm \infty} A^*_0(\xi', \xi_n)  = I. \\
\end{equation}
\end{theorem}
\begin{proof}
A detailed proof of this theorem is given in Appendix \ref{AppShamir}.\\
\end{proof}

\begin{remark}
Note that in \eqref{zetaqdefn}, the definition of $\zeta_q, \, q=1, \dots, N$ includes a multiplicative factor of (-1) not given in \cite{Sh}.
\end{remark}

\begin{remark} \label{delta0defn}
Since we are assuming that for every non-zero $z \in \mathbb{C}$ we have $-\pi < \arg z \leq \pi$, it follows immediately that 
\begin{equation*}
- 1/2  \leq \operatorname{Re } \zeta_j < 1/2,  \quad j = 1, \dots , N.
\end{equation*}
and hence
\begin{equation} \label{delta0}
\delta_0: = \min_{1 \leq j,k \leq N} ( 1 -  \operatorname{Re } \zeta_k +  \operatorname{Re } \zeta_j) >  0. \\
\end{equation}
\end{remark}

\newpage
\section{Statement of results}
\begin{theorem} \label{WrSplitRalg}
For $r=0,1,2, \dots, W^r(\mathbb{R})$ is a splitting R-algebra. 
\end{theorem}
Our second result considers the factorization of an elliptic matrix-valued function of degree $\mu$, and it confirms the isotropic case of Lemma 1.9, p.\,60 \cite{Shar}.
\begin{theorem} \label{ESLemma1.9}
Let $r := [n/2] +1$. Suppose that $A \in C^{r+3}_{N \times N}(\mathbb{R}^n)$ is a matrix-valued function which is homogeneous of degree $\mu$ and elliptic. Then, for fixed $\omega \in \mathbb{S}_{n-2}$, 
\begin{equation*}
A_\omega(\xi) = A(|\xi^\prime| \omega_1, \dots, |\xi^\prime| \omega_{n-1}, \xi_n) \end{equation*}
admits the factorization
\begin{equation*}
A_\omega(\xi) = (\xi_n - i |\xi^\prime|)^{\mu/2} A^-_\omega(\xi) \,  D(\omega, \xi) \,A^+_\omega(\xi) (\xi_n + i |\xi^\prime|)^{\mu/2},
\end{equation*}
where $(A^-_\omega(\xi))^{\pm 1}$ and $(A^+_\omega(\xi))^{\pm 1}$are homogeneous matrix-valued functions of order $0$ that, for fixed $\xi' \not = 0$, satisfy estimates of the form
\begin{equation} \label{DqApmestimate}
\sum_{0 \leq q \leq r} \operatorname{ess} \sup_{{\xi_n} \in \mathbb{R}} | \xi^q_n D^q_{\xi_n}  (A^{\pm}_\omega(\xi', \xi_n))_{j,k}| < + \infty, \quad 1 \leq j,k \leq N. 
\end{equation}
Further, they have analytic extensions, with respect to $\xi_n$, in the lower half plane and the upper half plane respectively. \\

$D(\omega, \xi)$ is a lower triangular matrix with elements
\begin{equation*}
\bigg (\dfrac{\xi_n - i|\xi^\prime|}{\xi_n + i|\xi^\prime|} \bigg )^{{\kappa_k(\omega)} + \zeta_k}
\end{equation*}
on its diagonal.  Its off-diagonal terms are homogeneous of degree $0$, and they satisfy an estimate of the form \eqref{DqApmestimate}. The integer
\begin{equation*}
\kappa(\omega) := \sum^N_{k=1} \kappa_k(\omega) =\dfrac{1}{2 \pi} \Delta \arg \det \big [(|\xi'|^2+\xi_n^2)^{-\mu/2} \, A_\omega(\xi', \xi_n)  \big ] \big |^{+\infty}_{\xi_n= - \infty} - \sum^N_{k=1} \operatorname{Re} \zeta_k
\end{equation*}
depends continuously on $\omega \in \mathbb{S}^{n-2}$. The partial sums $\sum^M_{k=1} \kappa_j(\omega), 1 \leq  M < N$, are upper semicontinuous;
\begin{equation*}
\zeta_k = - \dfrac{\log \lambda_j}{ 2 \pi i} \text{  for  } \sum^{j-1}_{\nu=1} m_\nu < k \leq 
\sum^{j-1}_{\nu=1} m_\nu, \quad k=1, \dots, N,
\end{equation*}
$\lambda_j$ are the eigenvalues of the matrix $A^{-1}(0, \dots ,0,+1)A(0, \dots, 0, -1)$ to which there correspond Jordan blocks of dimension $m_j$.
\end{theorem}

\newpage
\section{Proof of the first result}
The objective of this section is to prove Theorem \ref{WrSplitRalg}. Let $\theta^\pm$ denote the characteristic functions of $\mathbb{R}^\pm$ respectively. 

\begin{lemma} \label{w0ralg}
The Wiener algebra $W(\mathbb{R})$ is an R-algebra.
\end{lemma}
\begin{proof}
An abbreviated proof of this lemma is given in Chapter 2, Section 4, pp.\,62-63  \cite{CG}. A more detailed proof is included here, both for completeness and to introduce some analysis that will be useful when considering the subalgebras $W^r(\mathbb{R})$ for $r \geq 1$. \\

We begin by showing that $W(\mathbb{R})$ contains all rational functions with poles off $\dot{\mathbb{R}}$. Firstly, we note the identities
\begin{equation*}
(t-z_+)^{-1} = \mathcal{F}_{x \to t} \big (\sqrt{2 \pi} \, i \, \theta^-(x) \, e^{-i z_+ x} \big ),  \quad \text{Im } z_+ > 0,
\end{equation*}
\begin{equation*}
(t-z_-)^{-1} = - \mathcal{F}_{x \to t} \big ( \sqrt{2 \pi} \, i \, \theta^+(x) \, e^{-i z_- x} \big ),  \quad \text{Im } z_- < 0,
\end{equation*}
where the functions $\theta^-(x) e^{-i z_+ x}$ and $ \theta^+(x) e^{-i z_- x} \in L_1(\mathbb{R})$.  Secondly, since all functions in $W(\mathbb{R})$ are bounded at infinity, any rational function in $W(\mathbb{R})$ must be such that the degree of the numerator must be less than or equal to the degree of the denominator. (In particular, non-constant polynomial functions are not included in $W(\mathbb{R})$.) Finally, the fact that $W(\mathbb{R})$ contains all rational functions with poles off $\dot{\mathbb{R}}$ now follows directly, because $W(\mathbb{R})$ is an algebra and we have the usual partial fraction decomposition over $\mathbb{C}$. \\

We now wish to show that rational functions with poles off $\dot{\mathbb{R}}$ are dense in $W(\mathbb{R})$. Suppose $f \in W(\mathbb{R})$ is arbitrary and $r  \in W(\mathbb{R})$ is rational.  By definition, we can write $f(t) = \widehat{g}(t) + c$ and $r(t) = \widehat{s}(t) + d$, where $g,s \in L_1(\mathbb{R})$ and $c,d \in \mathbb{C}$. Let $C^\infty_c(\mathbb{R})$ denote the set of smooth functions with compact support in $\mathbb{R}$. Then,  $C^\infty_c(\mathbb{R})$ is dense in $L_1(\mathbb{R})$ and 
\begin{align*}
\| f - r \|_W & := \| g -s \|_{L_1} + |c - d| \\
& \leq \| g - h \|_{L_1} + \| h - s \|_{L_1} + |c - d | \quad \text{where  } h \in C^\infty_c(\mathbb{R}) \\
& =  \| g - h \|_{L_1} + \| \theta^+ h + \theta^- h - \theta^+ s - \theta^- s \|_{L_1} \quad (\text{taking  } d=c) \\
& \leq  \| g - h \|_{L_1} + \| \theta^+ h  - \theta^+ s  \|_{L_1} +  \| \ \theta^- h -  \theta^- s \|_{L_1}.
\end{align*}
Of course, the approximations to $\theta^+ h$ and $\theta^- h$, by $\theta^+ s$ and $\theta^- s$ respectively, are independent but similar. Hence, to prove that $W(\mathbb{R})$ is an R-algebra, it is enough for us to show that we can approximate $\theta^+(x) h(x)$, where $h \in C^\infty_c(\mathbb{R})$, arbitrarily closely in the $L_1(\mathbb{R})$ norm by a function $\theta^+ (x) s(x)$ such that $\widehat{\theta^+ s}$ is rational and has no poles in the upper half plane. \\

For $x \geq 0$, we let $y=e^{-x}$ and define
\[
 \psi(y) :=
  \begin{cases}
   h(-\log(y))/y & \text{if } y \in (0,1] \\
   0       & \text{if } y=0.
  \end{cases}
\]
Since $h(x)$ has compact support, $\psi(y)$ is identically zero in some interval $[0, \nu)$, where $\nu > 0$. Thus, by construction, $\psi(y) \in C^\infty [0,1]$.\\

Hence, given any $\epsilon > 0$, we can choose a Bernstein polynomial, see \cite{Lo}, $(B_M \psi)(y)$, of degree $M = M(\epsilon)$, such that
\begin{align*}
& \sup_{y \in [0,1]} | \psi(y) - (B_M \psi)(y) | < \epsilon \\
\implies & \sup_{y \in [0,1]} | \psi(y) - \sum^M_{k=0} b_k y^k | < \epsilon \quad \text{for certain } b_k \in \mathbb{C}, k=0,1,2, \dots, M \\
\implies & \sup_{x \in [0,\infty)} | h(x) e^{x} - \sum^M_{k=0} b_k e^{-kx} | < \epsilon. 
\end{align*}
We let $S(x) = \sum^M_{k=0} b_k e^{-kx}$ and observe, therefore, that our proposed approximant to $\theta^+ h(x)$ is $\theta^+ S(x) e^{-x}$. \\

Of course, the Fourier transform of $\theta^+ S(x) e^{-x}$ is a rational function with no poles in the upper half-plane, since for $k=1,2,3,  \dots$ we have
\begin{equation*}
\widehat{\theta^+ e^{-k x}} = \dfrac{i}{\sqrt{2 \pi}} \,  \dfrac{1}{t+i k}.
\end{equation*}

Finally, we take $\theta^+ s(x):= \theta^+ S(x) e^{-x}$ and then
\begin{align*}
\| \theta^+ h - \theta^+ s(x) \|_{L_1} & = \int^\infty_0 | h(x) - S(x) e^{-x} | \, dx \\
& = \int^\infty_0 | h(x)e^x  - S(x)  | \, e^{-x}dx \\
& \leq \epsilon \int^\infty_0 e^{-x}dx \\
& = \epsilon.
\end{align*}
This completes the proof that $W(\mathbb{R})$ is an R-algebra. \\
\end{proof}

\begin{remark} \label{ghatC+}
Suppose now that $f = \widehat{g} \in W(\mathbb{R})$. From the proof of the above lemma, we can show that  $\widehat{\theta^+ g} \in C^+(\dot{\mathbb{R}})$. (See section \ref{BActsfuns}.) Indeed, applying inequality \eqref{eq:hatnorm} we have 
\begin{equation*} 
\| \widehat{\theta^+ g} - \widehat{\theta^+ s(x)} \|_\infty \leq \| \theta^+ g - \theta^+ s(x) \|_{L_1}.
\end{equation*}
Since $\widehat{\theta^+ s(x)} \in R^+(\dot{\mathbb{R}})$, we immediately have $\widehat{\theta^+ g} \in C^+(\dot{\mathbb{R}})$, because $C^+(\dot{\mathbb{R}})$ is the closure of $R^+(\dot{\mathbb{R}})$ with respect to the supremum norm. It follows in an exactly similar way that  $\widehat{\theta^- g} \in C^-(\dot{\mathbb{R}})$. \\
\end{remark}

\begin{lemma} \label{Wsplits}
The Wiener algebra $W(\mathbb{R})$ splits.
\end{lemma}
\begin{proof}
An abbreviated proof of this lemma is given in Chapter 2, Section 4, p.\,63 \cite{CG}.  A more detailed proof is included here for completeness. \\

Our method of proof is a direct construction. Suppose $f = \widehat{g} + c \in W(\mathbb{R})$ then, since $g = \theta^+g + \theta^- g$, we have
\begin{align*}
f & = \widehat{\theta^+ g} + \widehat{\theta^- g} + c \\
&= \big ( \widehat{\theta^+ g} + c_+ \big ) + \big ( \widehat{\theta^- g} + c_-\big )
\end{align*}
where $c= c_+ + c_-$, and $c_-$ is chosen such that
\begin{equation*}
(\widehat{\theta^- g})(-i) + c_- =0. \\
\end{equation*}

But since $g \in L_1(\mathbb{R})$, we have $\theta^\pm g \in L_1(\mathbb{R})$. Moreover, from Remark \ref{ghatC+}, we have $\widehat{\theta^\pm g} \in C^\pm(\dot{\mathbb{R}})$ and thus
\begin{equation*}
\widehat{\theta^\pm g} \in W(\mathbb{R}) \cap C^\pm(\dot{\mathbb{R}}).
\end{equation*}
In other words, we have the required decomposition, and thus
\begin{equation*}
W(\mathbb{R}) = W^+(\mathbb{R}) \oplus \mathring{W}^-(\mathbb{R})
\end{equation*}
where $\mathring{W}^-(\mathbb{R}) = \{ h \in W^-(\mathbb{R}) : h(-i)=0 \}$.
This completes the proof that $W(\mathbb{R})$ splits. \\
\end{proof}

\begin{remark} \label{mapWW}
For any $\varphi \in \mathcal{S}(\mathbb{R})$, we now define three integral operators: \begin{equation*}
\Pi^\pm \varphi (t) = \dfrac{(\pm 1)}{2 \pi i} \lim_{\epsilon \to 0} \int^\infty_{-\infty} \dfrac{\varphi (\tau) d\tau}{\tau - (t \pm i \epsilon)} ; \quad S_\mathbb{R}  \varphi (t) = \dfrac{1}{\pi i} \int^\infty_{-\infty} \dfrac{ \varphi(\tau)}{\tau -t} d\tau. \\
\end{equation*}
For more details see \cite{Es} and \cite{Ga}. Each of these operators is bounded on $S(\mathbb{R})$. Moreover, see Chapter II Section 5, pp.\,70-71 \cite{Es}, 
\begin{equation*}
\Pi^\pm \widehat{\varphi} = \widehat{\theta^\pm \varphi} .
\end{equation*}

But since $S(\mathbb{R})$ is dense in $W_0(\mathbb{R}) := \{ f \in W(\mathbb{R}) : f = \widehat{g}, \, \, g \in L_1(\mathbb{R}) \}$, each of the singular integral operators can be extended, by continuity, to a bounded operator on $W_0(\mathbb{R})$. \\

Finally, we have the well-known formulae
\begin{equation*}
\Pi^+ + \Pi^- = I; \quad \Pi^+ = \frac{1}{2}(I + S_{\mathbb{R}}); \quad \Pi^- = \frac{1}{2}(I - S_{\mathbb{R}}). \\
\end{equation*}
\end{remark}

\begin{lemma}
For $r=1,2, 3, \dots, \, W^r(\mathbb{R})$ is a Banach algebra with a norm that is equivalent to the norm
\begin{equation*}
\| f \|_{W^r} = \| f \|_W + \sum^r_{k=1} \| (1-it)^k D^k f(t) \|_W.
\end{equation*}
\end{lemma}
\begin{proof}
The proof that $W^r(\mathbb{R})$ is a Banach algebra is straightforward. However, as an illustration, we will prove that given $f_1, f_2 \in W^r(\mathbb{R})$, the product $f_1 f_2 \in W^r(\mathbb{R})$ and $\| f_1 f_2 \|_{W^r} \leq C_r \| f_1 \|_{W^r} \| f_2 \|_{W^r}$, for some constant $C_r$ that only depends on $r$.\\

The existence of a norm $\| \cdot \|'_{W^r}$ equivalent to $\| \cdot \|_{W^r}$ and such that $\| f_1 f_2 \|'_{W^r} \leq  \| f_1 \|'_{W^r} \| f_2 \|'_{W^r}$ is then guaranteed by Theorem 10.2, p.\,246 \cite{RuFA}. \\

Suppose $f_1, f_2 \in W^r(\mathbb{R})$. Then, for any integer $p$ satisfying $1 \leq p \leq r$, 
\begin{equation*}
(1-it)^p D^p_t [f_1(t)f_2(t)] = \sum^p_{k=0} \binom{p}{k} [(1-it)^k D^k f_1] [(1-it)^{p-k} D^{p-k} f_2].
\end{equation*}
We assume that $W(\mathbb{R})$ is a Banach algebra and therefore, $f_1f_2 \in W(\mathbb{R})$ and $ (1-it)^p D^p [f_1(t)f_2(t)] \in W(\mathbb{R})$. Hence, $f_1 f_2 \in W^r(\mathbb{R})$ as required. \\

By definition,  $\| f_1 f_2 \|_{W^r}$
\begin{align*}
&= \| f_1 f_2 \|_W + \sum^r_{k=1} \| (1-it)^k D^k [f_1 f_2] \|_W \\
&=  \| f_1 f_2 \|_W + \sum^r_{k=1} \| \sum^k_{j=0} \binom{k}{j} [(1-it)^j D^j f_1] [(1-it)^{k-j} D^{k-j} f_2] \|_W \\
& \leq  \| f_1\|_W \| f_2 \|_W + \sum^r_{k=1}  \sum^k_{j=0} \binom{k}{j} \|(1-it)^j D^j f_1\|_W \|(1-it)^{k-j} D^{k-j} f_2 \|_W \\
& \leq C_r \| f_1 \|_{W^r} \| f_2 \|_{W^r},
\end{align*}
where the strictly positive constant $C_r$ only depends on the integer $r$. This completes the proof of the lemma. \\
\end{proof}

We now show that $W^r(\mathbb{R})$ splits. To do this, we will need two intermediate lemmas.

\begin{lemma} \label{Commute}
Suppose $f(t), \, Df(t) \in W(\mathbb{R})$ and $\lim_{t \to \pm \infty} f(t) = 0$. Then $\Pi^\pm Df(t) = D \Pi^\pm f(t)$.
\end{lemma}
\begin{proof} 
From Chapter I, Section 4.4, p.\,31 \cite{Ga}, we have
\begin{equation*}
D S_{\mathbb{R}} f(t) = S_{\mathbb{R}} D f(t)
\end{equation*}
But, from Remark \ref{mapWW} we have $\Pi^\pm = \frac{1}{2} (I \pm S_{\mathbb{R}})$ respectively, and so
\begin{equation*}
D \Pi^\pm f(t) = \Pi^\pm D f(t). \\
\end{equation*}
\end{proof}

\begin{lemma} \label{tPi}
Suppose $f(t), \, tf(t) \in W(\mathbb{R})$ and $\lim_{t \to \pm \infty} f(t) = 0$. Let $[tI, \Pi^\pm]$ denotes the commutator of $tI$ and $\Pi^\pm$. Then $[tI, \Pi^\pm]f \in \mathbb{C}$.
\end{lemma}
\begin{proof}
Suppose $f(t), \, tf(t) \in W(\mathbb{R})$. Then 
\begin{align*}
[tI, \Pi^+]f &= (t I\Pi^+ - \Pi^+ tI)f \\
&= \frac{1}{2} (t(I + S_\mathbb{R}) - (I  + S_\mathbb{R})tI)f \quad \text{(by Remark } \ref{mapWW}) \\
&=  \frac{1}{2} (t S_\mathbb{R} - S_\mathbb{R}t)f \\
&= \frac{t}{\pi i} \int^\infty_{-\infty} \frac{f(\tau) d\tau}{\tau - t} - \frac{1}{\pi i} \int^\infty_{-\infty} \frac{\tau f(\tau) d\tau}{\tau - t} \\
&= \frac{1}{\pi i} \int^\infty_{-\infty} \frac{(t-\tau)f(\tau) d\tau}{\tau - t} \\
&= \frac{(-1)}{\pi i} \int^\infty_{-\infty} f(\tau) d\tau \\
& \in \mathbb{C}
\end{align*}
Finally, we note that $[tI, \Pi^-] = [tI, I - \Pi^+] = [tI, I]-[tI, \Pi^+] = 0 - [tI, \Pi^+]$ and hence  $[tI, \Pi^-]\in \mathbb{C}$. This completes the proof of the lemma. \\
\end{proof}

\begin{lemma} \label{WrSplits}
For $r=0,1,2, \dots$ the algebra $W^r(\mathbb{R})$ splits.
\end{lemma}
\begin{proof}
Suppose $f(t) \in W^r(\mathbb{R})$ for some nonnegative integer $r$. Since $f(t) \in W(\mathbb{R})$, it is enough to consider the case where $\lim_{t \to \pm \infty} f(t) = 0$. Moreover, by Remarks  \ref{ghatC+} and \ref{mapWW}, we can write
\begin{equation*} 
f(t) = \Pi^+ f(t) + \Pi^- f(t), \quad \Pi^\pm f \in W(\mathbb{R}) \cap C^\pm(\dot{\mathbb{R}}) .
\end{equation*}

Thus, to complete the proof, we must show that $\Pi^\pm f(t) \in W^r(\mathbb{R})$. That is, we must prove that for $k = 0,1, \dots r$ we have $(1-it)^k D^k \Pi^\pm f(t) = i^{-k}(t+i)^k D^k \Pi^\pm f(t) \in W(\mathbb{R})$. \\

We now proceed by induction on $r$. Our inductive hypothesis is that for any $f \in W^r(\mathbb{R})$, we have $(t+i)^r D^r \Pi^\pm f(t) =  (\Pi^\pm (t+i)^r D^r f(t) + c) \in W(\mathbb{R})$. We have previously proved this result for $r=0$. Suppose that the inductive hypothesis holds for $k=0, \dots , (r-1)$.  \\

From Lemma \ref{Commute}, 
\begin{align*}
(t+i)^r D^r \Pi^\pm f & = t \cdot (t+i)^{r-1} D^r \Pi^\pm f  + i \cdot (t+i)^{r-1} D^r \Pi^\pm f \\
& = t \cdot (t+i)^{r-1} D^{r-1} \Pi^\pm (Df)  + i \cdot (t+i)^{r-1} D^{r-1} \Pi^\pm (Df). 
\end{align*}
But since $Df \in W^{r-1}(\mathbb{R})$, applying the inductive hypothesis gives
\begin{align*}
(t+i)^r D^r \Pi^\pm f & = t \cdot \Pi^\pm (t+i)^{r-1} D^{r-1} (Df)  + i \cdot \Pi^\pm (t+i)^{r-1} D^{r-1}  (Df)  + c \\
 & = t \cdot \Pi^\pm (t+i)^{r-1} D^r f  + i \cdot \Pi^\pm (t+i)^{r-1} D^r f  + c. 
\end{align*}
Hence, using Lemma \ref{tPi} (applied to $(t+i)^{r-1}D^r f$) gives
\begin{align*}
(t+i)^r D^r \Pi^\pm f &=  \Pi^\pm t (t+i)^{r-1} D^r f   + \Pi^\pm i(t+i)^{r-1} D^r f  + c^\prime\\
&=  \Pi^\pm (t+i)^{r} D^{r} f  + c^\prime \\
& \in W(\mathbb{R}).
\end{align*}
This completes the proof by induction.  So, finally, for $k = 0,1, \dots r$, we have $(1-it)^k D^k \Pi^\pm f(t) \in W(\mathbb{R})$ and thus, for $r=0,1,2, \dots$, the algebra $W^r(\mathbb{R})$ splits. \\
\end{proof}

Our final objective in this section is to show that $W^r(\mathbb{R})$ is an R-algebra for $r=1,2, 3, \dots$, noting that in Lemma \ref{w0ralg}, we have proved this result for the special case $W(\mathbb{R})$, corresponding to $r=0$. \\

In Appendix \ref{AppWrApprox} we show that the Fourier transforms of smooth functions with compact support and which are zero in a neighbourhood of $x=0$, are dense in the space $W^r(\mathbb{R})$. Then, proceeding analogously to Lemma \ref{w0ralg}, it is enough for us to show that we can approximate $\widehat{\theta^+ h}$, where $h \in C^\infty_c(\mathbb{R})$ and is zero near $0$, arbitrarily closely in the $W^r(\mathbb{R})$ norm by a function  $\widehat{\theta^+ s}$ that is rational and has no poles in the upper half plane. \\

As previously, for $x \geq 0$, we set $y = e^{-x}$ and define
\[
 \psi(y) =
  \begin{cases}
  h(- \log (y)) /y & \text{if } y \in (0,1] \\
   0       & \text{if } y=0.
  \end{cases}
\]
Since $h(x)$ has compact support, $\psi(y)$ is identically zero in some interval $[0, \nu)$, where $\nu > 0$. Thus, by construction, $\psi(y) \in C^\infty [0,1]$. \\

\begin{remark}
The motivation for choosing the Bernstein polynomial, $(B_M \psi)(y)$, see \cite{Lo}, as the approximant to $\psi(y)$ in Lemma \ref{w0ralg}, is that we can \textit{simultaneously} choose $M = M(\epsilon)$ such that for $1 \leq j \leq r$
\begin{equation*}
\sup_{y \in [0,1]} | \psi(y) - (B_M \psi) (y) | < \epsilon \quad \text{and } \sup_{y \in [0,1]} | D^j_y\psi (y) - D^j_y (B_M \psi)(y) | < \epsilon. \\
\end{equation*} 
\end{remark} 

Given that $y = e^{-x}$, we can consider $\psi(y)$ in terms of $x$, as given by the equation $\psi(y) = e^x h(x)$. The following lemma expresses the derivatives of $\psi(y)$ in terms of the derivatives of $h(x)$.

\begin{lemma} \label{Djpsi}
\begin{equation*}
D^j_y \psi(y) = (-1)^j e^{(j+1)x} (D_x + 1) \dots (D_x+j) h(x) \text{ for } j = 1,2, \dots.
\end{equation*}
\end{lemma}
\begin{proof}
Note that, by definition, $y = e^{-x}$ and $\psi(y) = e^x h(x)$. We use proof by induction on $j$. \\

Suppose $j=1$.   Then $D_x \psi(y) = D_y \psi(y) \cdot (dy/dx)$ and hence
\begin{equation*}
D_y \psi(y) = - e^x D_x (e^x h) = - e^x ( e^x h + e^x D_x h ) = (-1)e^{2x}(D_x + 1)h,
\end{equation*}
completing the first step of the inductive proof. \\

Now suppose the result is true for $j=m$. Then, by the inductive hypothesis
\begin{equation*}
D_x [ D^m_y \psi(y) ] = D_x [ (-1)^m e^{(m+1)x} (D_x + 1) \dots (D_x+m) h(x) ].
\end{equation*}
Hence
\begin{align*}
D^{m+1}_y \psi(y)(dy/dx) & = (-1)^m (m+1) e^{(m+1)x} (D_x + 1) \dots (D_x+m) h(x) \\
& \quad + (-1)^m e^{(m+1)x} D_x(D_x + 1) \dots (D_x+m) h(x). 
\end{align*}
Therefore
\begin{align*}
D^{m+1}_y \psi(y) & = (-1)^{m+1} e^{(m+2)x}[m+1+D_x] (D_x + 1) \dots (D_x+m) h(x) \\
& = (-1)^{m+1} (m+1) e^{(m+2)x} (D_x + 1) \dots (D_x+m) (D_x +m+1)h(x),
\end{align*}
proving the result for $j = m+1$. This completes the proof by induction. \\
\end{proof}

Motivated by Lemma \ref{Djpsi}, for $j=0,1,2 \dots$, we now define:
\[
 h_j(x) =
  \begin{cases}
   (D_x+1) \dots (D_x + j) h(x) & \text{if } j > 0 \\
   h(x)       & \text{if } j=0.
  \end{cases}
\]
Hence, we can write 
\begin{equation} \label{Djpsiy}
D^j_y \psi(y) = (-1)^j e^{(j+1)x} h_j(x) \quad j=0,1,2 \dots. \\
\end{equation}

In exactly the same way, given $y = e^{-x}$ and $(B_M\psi)(y) = S(x)$ we define $T(x) = S(x) e^{-x}$. Hence, $(B_M\psi)(y) = e^x T(x)$ and
\begin{equation*}
D^j_y (B_M\psi)(y) = (-1)^j e^{(j+1)x} (D_x + 1) \dots (D_x+j) T(x) \text{ for } j = 1,2, \dots.
\end{equation*}
Analogously, for $j=0,1,2 \dots$ we define:
\[
 T_j(x) =
  \begin{cases}
   (D_x+1) \dots (D_x + j) T(x) & \text{if } j > 0 \\
   T(x)       & \text{if } j=0.
  \end{cases}
\]
Hence, we can similarly write 
\begin{equation} \label{DjBMy}
D^j_y (B_M\psi)(y) = (-1)^j e^{(j+1)x} T_j(x) \quad j=0,1,2 \dots,
\end{equation} 
and we can now express our approximations in terms of the variable $x$.\\

\begin{remark}
Using equations \eqref{Djpsiy} and \eqref{DjBMy}, we can now reformulate the Bernstein polynomial, $(B_M \psi)(y)$, approximations to $\psi(y)$ and its derivatives as
\begin{equation} \label{g0approx}
\sup_{x \in [0,\infty)} | e^x h_0(x) - e^x T_0(x)| < \epsilon.
\end{equation} 
and for $1 \leq j \leq r$
\begin{align}  \label{gjapprox}
\sup_{x \in [0,\infty)} & | e^{(j+1)x} h_j(x) - e^{(j+1)x} T_j(x) | < \epsilon. \\
\nonumber
\end{align}
\end{remark}


\begin{lemma} \label{WrRalgebra}
For $r=1,2,3, \dots \, W^{r}(\mathbb{R})$ is an R-algebra.
\end{lemma}
\begin{proof}
Our proposed approximant to $\theta^+ h(x)$ is $\theta^+ S(x)e^{-x}$. From Appendix \ref{AppWrApprox}, to show convergence to $\widehat{\theta^+ h}$ in $\| \cdot \|_{W^r}$, it is sufficient to show convergence to $\theta^+ h, x^k (\theta^+ h)$ and $D^j_x(x^k  (\theta^+ h))$ in $\| \cdot \|_{L_1}$ for all $1 \leq j \leq k \leq r$. \\

Of course, one important consequence of the fact that our smooth function $h$ is zero in a neighbourhood of $0$ is that it implies that $\theta^+h$ is also smooth. \\

We have already seen in Lemma \ref{w0ralg} that
\begin{equation*}
\| \theta^+ h(x) - \theta^+ S(x)e^{-x} \|_{L_1} < \epsilon.
\end{equation*}

Similarly, for $1 \leq k \leq r$ we have 
\begin{align*}
\| \theta^+x^k h(x) - \theta^+ x^k S(x) e^{-x} \|_{L_1} & = \int^\infty_0 | x^k h(x) - x^k S(x) e^{-x} | \, dx \\
& = \int^\infty_0 | e^x h(x)  - S(x)  | \, x^k e^{-x} \, dx \\
& = \int^\infty_0 | e^x h_0(x)  - e^x T_0 (x)  | \, x^k e^{-x} \, dx \\
& \leq \epsilon \int^\infty_0 x^k e^{-x}dx \quad \text{by } \eqref{g0approx} \\
& = (k!) \,  \epsilon, \quad \text{since } \int^\infty_0 x^k e^{-x}dx=k!. 
\end{align*}

Suppose that $j  \geq 1$. Clearly, there exist constants $\{ c_l : 0 \leq l \leq j \}$, \textit{that only depend on $j$}, such that
\begin{equation*}
D^j_x h = \sum^j_{l=0} c_l h_l \, ; \quad D^j_x T = \sum^j_{l=0} c_l T_l,
\end{equation*}
where $h_0 = h$, and $h_l = (D_x + 1) \dots (D_x + l)h$ for $l > 0$, and $T_0 = T = S e^{-x}$, and $T_l = (D_x + 1) \dots (D_x + l)T$ for $l > 0$. \\

Hence, for $1 \leq j \leq k \leq r$
\begin{align*}
\| \theta^+ x^k D^j_x h - \theta^+ x^k D^j_x (S e^{-x}) \|_{L_1} & = \| \theta^+ x^k \sum^j_{l=0} c_l (h_l  - T_l )\|_{L_1} \\
& \leq \sum^j_{l=0} |c_l|  \cdot \| \theta^+ x^k (h_l  - T_l )\|_{L_1} \\
& = \sum^j_{l=0} |c_l|  \int^\infty_0 | x^k (h_l(x)  - T_l(x) ) | \, dx \\
& = \sum^j_{l=0} |c_l|  \int^\infty_0 | e^{(l+1)x}h_l(x)  - e^{(l+1)x}T_l(x) | e^{-(l+1)x}x^k \, dx \\
&  \leq \epsilon \sum^j_{l=0} |c_l|  \int^\infty_0 e^{-(l+1)x}x^k \, dx \quad \text{by } \eqref{gjapprox}\\
&  \leq \epsilon \sum^j_{l=0} |c_l| \, \dfrac{k!}{(l+1)^{k+1}}. 
\end{align*}

Therefore, $W^{r}(\mathbb{R})$ is an R-algebra, as required. 
\end{proof}
\newpage
\section{Proof of the second result}
The objective of this section is to prove Theorem \ref{ESLemma1.9}. In determining certain asymptotic estimates for matrices arising during factorization, we follow the approach of Duduchava \cite{Du}. (For full details, see Appendix \ref{AppMatrixFactors}.)

\begin{proof}
We begin by defining 
 \begin{equation} \label{A0DefinitionStart} 
 A_0(\xi',\xi_n) := (|\xi'|^2+|\xi_n|^2)^{-\mu/2} \, A(\xi',\xi_n). 
 \end{equation}
For fixed $\xi' \not =0$, we set
\begin{equation*}
\omega := \dfrac{\xi'}{|\xi'|}; \quad t := \dfrac{\xi_n}{|\xi'|}.
\end{equation*}
From Theorem \ref{MainResultCase3},  for fixed $\omega \in S^{n-2}$,
\begin{equation*}
A^*_0(\omega, t)= (t-i)^{-\zeta} B_-(t)  h^{-1} c A_0(\omega, t) h B^{-1}_+(t) (t+i)^{\zeta} \in W^{r+2}_{N \times N}(\mathbb{R}).
\end{equation*}

Moreover,  from Lemmas \ref{WrSplits} and \ref{WrRalgebra}, $W^{r+2}(\mathbb{R})$ is a splitting R-algebra. Hence, by Theorem \ref{thm:ralg}, the matrix $A^*_0(\omega,t)$ admits a right standard factorization. \\

Therefore, we can write $c A_0(\omega, t)$
\begin{align*}
&= hB^{-1}_-(t) (t-i)^{\zeta}A^*_0(\omega, t)(t+i)^{-\zeta}B_+(t)h^{-1} \\
&= hB^{-1}_-(t) (t-i)^{\zeta} \bigg [ (A^*_-(\omega, t))^{-1} \text{diag} \bigg ( \dfrac{t-i}{t+i} \bigg )^{{\kappa(\omega)}} A^*_+(\omega, t) \bigg ] (t+i)^{-\zeta}B_+(t)h^{-1},
\end{align*}
where the factors $(A^*_\pm)^{\pm 1} \in W^{r+2}_{N \times N}(\mathbb{R})$, and have analytic extensions, with respect to $\xi_n$, to the lower half-plane and the upper half-plane respectively. \\

Moreover, see p.\,37 \cite{GetAl}, since $\lim_{t \to \pm \infty} A^*_0(\omega,t) =I$, there exist factors $A^*_\pm \in W^{r+2}_{N \times N}( \mathbb{R})$ such that
\begin{equation} \label{limA*pm}
\lim_{t \to \pm \infty} A^*_\pm (\omega,t) = I. \\
\end{equation}
We now define
\begin{equation} \label{A1pmdefinition}
A^\pm_1(\omega, t) :=  (t \pm i)^\zeta \, A^*_\pm(\omega, t) \, (t\pm i)^{-\zeta}.
\end{equation}

Hence, as diagonal matrices commute,
\begin{equation*}
c A_0(\omega, t) = hB^{-1}_-(t) (A^-_1(\omega, t))^{-1} \text{diag} \bigg ( \dfrac{t-i}{t+i} \bigg )^{{\kappa(\omega)} + \zeta} A^+_1(\omega, t) B_+(t)h^{-1}. \\
\end{equation*}

From equation \eqref{A1pmdefinition}
\begin{equation*}
(A^\pm_1)_{j,k} = (t+i)^{\zeta_j - \zeta_k} \, (A^*_\pm)_{j,k}.
\end{equation*}
Suppose $j \not = k$. Then, from Lemma \ref{Apm1jnotk},
\begin{equation*}
\lim_{t \to \pm \infty} (A^\pm_1(w,t))_{j,k} = 0 \quad (j \not = k).
\end{equation*}
Given this result for the off-diagonal terms of $A^\pm_1(w,t)$ and equations \eqref{limA*pm} and \eqref{A1pmdefinition}, we have
\begin{equation*}
\lim_{t \to \pm \infty} A^\pm_1(w,t) = I. \\
\end{equation*}

Further, if we set 
\begin{equation*}
A^\pm_0(\omega, t) := A^\pm_1(\omega, t) \,  B_\pm(t) h^{-1},
\end{equation*}
then we can write
\begin{equation*}
c A_0(\omega, t) =  (A^-_0(\omega, t))^{-1} \, \text{diag} \bigg ( \dfrac{t-i}{t+i} \bigg )^{{\kappa(\omega)} + \zeta} A^+_0(\omega, t).
\end{equation*}
Now, by definition,
\begin{align*}
A^\pm_0 &= A^\pm_1 B_\pm h^{-1} \\
&=  \big [ (A^\pm_1 - I) + I \big] B_\pm h^{-1} \\
&=  B_\pm \big [ B^{-1}_\pm(A^\pm_1 - I) B_\pm + I \big]  h^{-1} \\
&= B_\pm A^\pm_2 h^{-1},
\end{align*}
where we now define
\begin{equation} \label{A2pmdefinition}
A^\pm_2(\omega, t) := B^{-1}_\pm(t) \, (A^\pm_1(\omega, t) - I) \, B_\pm(t)  + I.
\end{equation}

\begin{remark}
We have already noted that the factors $(A^*_\pm)^{\pm 1}$ have analytic extensions, with respect to $\xi_n$, to the lower half-plane and the upper half-plane respectively. From definitions \eqref{A1pmdefinition} and \eqref{A2pmdefinition}, it is immediately clear that this property is also shared by the factors $(A^\pm_1)^{\pm 1}$ and $(A^\pm_2)^{\pm 1}$. \\
\end{remark}

\begin{remark}
From Lemma \ref{Apm2inWr}
\begin{equation*}
 [(A^\pm_2)^{\pm 1}]_{j,k} \in W^{r}(\mathbb{R}) \,\, \text{for } 1 \leq j,k \leq N.
\end{equation*}
In particular, each element of the matrices $(A^\pm_2)^{\pm 1}$ satisfies a condition of the form:
\begin{equation*} 
\sum_{0 \leq q \leq r} \operatorname{ess} \sup_{{\xi_n} \in \mathbb{R}} | \xi^q_n D^q_{\xi_n}  (A^\pm_2(\xi', \xi_n))_{j,k}| < + \infty. \\
\end{equation*}
\end{remark}

Finally, we have the required factorization, namely
\begin{align} \label{A0FinalFact}
c A_0(\omega, t) & = h (A^-_2)^{-1} B^{-1}_- \text{diag} \bigg ( \dfrac{t-i}{t+i} \bigg )^{{\kappa(\omega)} + \zeta} B_+A^+_2 h^{-1} \nonumber \\
& = h (A^-_2)^{-1} d(\omega, t) A^+_2 h^{-1},
\end{align}
where
\begin{equation*}
d(\omega, t) := B^{-1}_-(t) \, \text{diag} \bigg ( \dfrac{t-i}{t+i} \bigg )^{{\kappa(\omega)} + \zeta} \, B_+(t). \\
\end{equation*}

\begin{remark} \label{Bdiagcommute}
Note that, by construction, the matrix-valued functions $B_\pm(\xi_n)$ commute with the diagonal matrix $(\xi_n \pm i )^\zeta$. [To see this, choose an arbitrary block $J_k(\lambda_k)$. On this block, $(\xi_n \pm i )^\zeta$ acts like a scalar, since the relevant components of the vector $\zeta$ are all equal to $- (\log \lambda_k)/(2 \pi i)$]. \\
\end{remark}

By Remark \ref{Bdiagcommute}, equation \eqref{Bpmdef} which defines $B_\pm(t)$, together with the properties of the blocks (see equations \eqref{Bab} and \eqref{B-a}), we can write
\begin{align*}
d(\omega, t) & = \text{diag} \bigg (\dfrac{t-i}{t+i} \bigg )^{{\kappa(\omega)} + \zeta} B^{-1}_-(t) B_+(t) \\
& = \text{diag} \bigg (\dfrac{t-i}{t+i} \bigg )^{{\kappa(\omega)} + \zeta} \,\, \text{diag } \bigg [ B^{m_1} \bigg ( \dfrac{1}{2 \pi i} \log \dfrac{t+i}{t-i} \bigg ) , \dots,  B^{m_l} \bigg ( \dfrac{1}{2 \pi i} \log\dfrac{t+i}{t-i} \bigg )  \bigg ].\\
\end{align*}

\begin{remark}
We note that, by definition, 
\begin{equation*}
t = \dfrac{\xi_n}{|\xi'|}  \quad \text{and} \quad \dfrac{t+i}{t-i} = \dfrac{\xi_n + i |\xi'|}{\xi_n - i |\xi'|}.
\end{equation*}
Hence, functions of $t$ or $(t+i)/(t-i)$ are homogeneous in the variable $\xi = (\xi', \xi_n)$. \\
\end{remark}

It remains to consider the sum and partial sums of the factorization indices. For fixed $\xi'$, our final factorization, see equation \eqref{A0FinalFact}, is
\begin{equation*}
c A_0(\omega, t) = h (A^-_2)^{-1} d(\omega,t) A^+_2 h^{-1}.
\end{equation*}
Hence, since $c,h$ are constant matrices and $\lim_{t \to \pm \infty} A^\pm_2 = I$, we have
\begin{align}  \label{Dargdetd1}
& \Delta \arg \det \big[ (|\xi'|^2 + |\xi_n|^2)^{-\mu/2} A_\omega(\xi',\xi_n) \big ]  \big |^{\xi_n=+\infty}_{\xi_n = -\infty} \nonumber \\
&=  \Delta \arg \det A_0(\xi', \xi_n)  \big |^{\xi_n=+\infty}_{\xi_n = - \infty} \quad ( \text{see equation } \eqref{A0DefinitionStart} ) \nonumber \\
& = \Delta \arg \det d(\omega,t) \big |^{t=+\infty}_{t = - \infty}.
\end{align}
Now $d(\omega,t)$ is a lower triangular matrix, and hence its determinant is the product of the entries on its main diagonal. Thus
\begin{align*}
\det d(\omega,t) & = \displaystyle \prod^N_{k=1} \bigg ( \dfrac{t-i}{t+i} \bigg)^{\kappa_k(\omega) + \zeta_k}.
\end{align*}
Therefore, see Chapter II Section 6, p.\,88 \cite{Es}.
\begin{equation} \label{Dargdetd2}
\dfrac{1}{2 \pi} \Delta \arg \det d(\omega,t) \big |^{t=+\infty}_{t = - \infty} = \sum^N_{k=1} \kappa_k(\omega) + \sum^N_{k=1} \operatorname{Re} \zeta_k.
\end{equation}
From equations \eqref{Dargdetd1} and \eqref{Dargdetd2} we have
\begin{equation*}
\dfrac{1}{2 \pi} \Delta \arg \det \big[ (|\xi'|^2 + |\xi_n|^2)^{-\mu/2} A_\omega(\xi',\xi_n) \big ]  \big |^{\xi_n=+\infty}_{\xi_n = -\infty} = \sum^N_{k=1} \kappa_k(\omega) + \sum^N_{k=1} \operatorname{Re} \zeta_k.
\end{equation*}
The remaining assertions in Theorem \ref{ESLemma1.9}, concerning the continuity of the sum and the semicontinuity of the partial sums of the factorization indices, can be found in Theorem 3.1 p.\,113 \cite{Sh}. 
\end{proof}

\newpage
\appendix 
\appendixpage
\addappheadtotoc

\section{Proof of key theorem from Shamir} \label{AppShamir}
We shall give a proof of Theorem \ref{MainResultCase3} in two steps:
\begin{align*}
\text{(i)} \quad &E_+^{-1}E_- \text{ is similar to } \text{diag}[\lambda_1, \dots , \lambda_N] \\
\text{(ii)} \quad &\text{The general case.} 
\end{align*}
Our overall approach will be to reduce the general case to the simpler case. \\
	
We begin by establishing some simple decay estimates.
\begin{lemma} \label{Decay}
Suppose that $A_0(\xi^\prime, \xi_n) \in C_{N \times N}^{r+3}(\mathbb{S}^{n-1})$ is a matrix-valued function which is homogeneous of degree $0$. Then, for fixed $\xi^\prime \not =0$,
\begin{equation} 
D^k_{\xi_n} [ A_0(\xi^\prime, \xi_n) - E_\pm ] = O ( |\xi_n|^{-k-1} ) \quad \xi_n \rightarrow \pm \infty; \quad 0 \leq k \leq (r+3),
\end{equation}
where these estimates are uniform for $\xi' \in \mathbb{S}^{n-2}$. \\
\end{lemma}
\begin{proof}
Suppose that $\xi_n \rightarrow \infty$. Then since $A_0$ is homogeneous of degree $0$,
\begin{align*}
A_0(\xi^\prime, \xi_n) - E_+ & = A_0(\xi^\prime \xi_n^{-1}, 1) - A_0(0, 1) \\
& = \sum^{n-1}_{j=1} \dfrac{\partial A_0}{\partial \xi_j}(0,1) \dfrac{\xi_j}{\xi_n} + O(|\xi_n|^{-2}) = O(|\xi_n|^{-1}), \quad \xi_n \rightarrow \infty .
\end{align*}
This completes the proof for $k=0$. \\

For $1 \leq k \leq (r+3)$ we can ignore the constant matrix $E_+$, and we readily obtain
\begin{equation*}
D^k_{\xi_n} A_0(\xi^\prime, \xi_n) = D^k_{\xi_n} A_0(\xi^\prime \xi_n^{-1}, 1 ) = O ( |\xi_n|^{-k-1} ), \quad \xi_n \rightarrow \infty. 
\end{equation*}
Of course, estimates for the case $\xi_n \rightarrow -\infty$ follow in exactly the same way. This completes the proof of the lemma. \\
\end{proof}

Let us consider the first step. We are assuming that the invertible matrix $E_+^{-1}E_-$ is similar to diag$[\lambda_1, \dots , \lambda_N]$. In this formulation the eigenvalues $\lambda_j, \, j=1, \dots, N$ are listed according to their multiplicity and, of course, are all non-zero. \\

We now define $\zeta = (\zeta_1, \dots , \zeta_N)$ by
\begin{equation} \label{zetadef}
\zeta_j = - (\log \lambda_j)/(2 \pi i)  \quad j = 1, \dots , N. \\
\end{equation}

\begin{remark}
The definition of $\zeta$, given by equation \eqref{zetadef}, includes a multiplicative factor $(-1)$ not shown in \cite{Sh}. As will be seen, this modification allows us to correct an error in the treatment of the discontinuity across the negative real axis. (See Lemma 4.2, \cite{Sh}.) \\
\end{remark}

\begin{lemma} \label{Case2}
Suppose that $A_0(\xi^\prime, \xi_n) \in C_{N \times N}^{r+3}(\mathbb{S}^{n-1})$ is a matrix-valued function which is homogeneous of degree $0$ and elliptic. Suppose further that for some invertible constant matrix $h_1$,
\begin{equation}
E= E_+^{-1}E_-  = h_1 \text{diag}\, [\lambda_1, \dots , \lambda_N] h^{-1}_1.
\end{equation}
If $\zeta_j = - (\log \lambda_j)/ 2 \pi i$, for $j=1, \dots N, \,\, \zeta = (\zeta_1, \dots \zeta_N)$ and $c:= A^{-1}_0(0, \dots, 0,1)$, then for fixed $\xi^\prime \not =0$,
\begin{equation*}
A^*_0(\xi', \xi_n) := (\xi_n - i)^{-\zeta}  h^{-1}_1 c A_0(\xi^\prime, \xi_n) h_1  (\xi_n + i)^{\zeta} \in W^{r+2}_{N \times N}(\mathbb {R}),
\end{equation*}
and
\begin{equation*} 
\lim_{\xi_n \to \pm \infty} A^*_0(\xi', \xi_n) = I. \\
\end{equation*}
\end{lemma}

\begin{proof}
By hypothesis, we have $E= E_+^{-1}E_-  = h_1 \text{diag} \, [\lambda_1, \dots , \lambda_N] h^{-1}_1$. If we define $\widetilde{A}_0(\xi^\prime, \xi_n) := h^{-1}_1 c A_0(\xi^\prime, \xi_n) h_1$, we may assume, without loss of generality,  that 
\begin{equation*}
E_+ = I; \quad E_- = \text{diag} [\lambda_1, \dots , \lambda_N].
\end{equation*} 

We define a new matrix-valued function
\begin{equation} \label{Astar2}
A^*_0(\xi^\prime, \xi_n) = (\xi_n - i)^{-\zeta} \widetilde{A}_0(\xi^\prime, \xi_n) (\xi_n + i)^{\zeta}
\end{equation}
Then, for $\xi_n > 0$, we can write
\begin{equation*}
A^*_0(\xi^\prime, \xi_n) = (\xi_n - i)^{-\zeta} [\widetilde{A}_0(\xi^\prime, \xi_n) - E_+] (\xi_n + i)^{\zeta} + (\xi_n - i)^{-\zeta} E_+ (\xi_n + i)^{\zeta}
\end{equation*}
and similarly for $\xi_n < 0$, we have
\begin{equation*}
A^*_0(\xi^\prime, \xi_n) = (\xi_n - i)^{-\zeta} [\widetilde{A}_0(\xi^\prime, \xi_n) - E_-] (\xi_n + i)^{\zeta} + (\xi_n - i)^{-\zeta} E_- (\xi_n + i)^{\zeta}
\end{equation*}
Since the matrices $(\xi_n - i)^{-\zeta}$ and $(\xi_n + i)^{\zeta}$ are diagonal, we can write a typical element of the first summand as
\begin{equation} \label{Firstterm}
(\xi_n - i)^{-\zeta_j} [\widetilde{A}_0(\xi^\prime, \xi_n) - E_\pm]_{jl} (\xi_n + i)^{\zeta_l} = O(|\xi_n|^{- \text{Re } \zeta_j -1 + \text{Re } \zeta_l} )= O(|\xi_n|^{-\delta_0})
\end{equation}
using Lemma \ref{Decay} and equation \eqref{delta0}. \\

Now suppose $1 \leq k \leq (r+3)$ and $\alpha \in \mathbb{C}$. Then $D^k_{\xi_n} (\xi_n \pm i)^\alpha = C_{\alpha, k} (\xi_n \pm i)^{\alpha - k}$, where $C_{\alpha, k}, \, k=1,2, \dots $ are certain constants. Moreover, from Lemma \ref{Decay}, $D^k_{\xi_n}  [\widetilde{A}_0(\xi^\prime, \xi_n) - E_\pm] = O ( |\xi_n|^{-k-1} )$ as $\xi_n \rightarrow \pm \infty$. Hence, for $k=0, 1, \dots, r+3$, we can write
\begin{align} \label{FirsttermD}
& D^k_{\xi_n} \big \{  (\xi_n - i)^{-\zeta_j} [\widetilde{A}_0(\xi^\prime, \xi_n) - E_\pm]_{jl} (\xi_n + i)^{\zeta_l} \big \} \nonumber \\
& = O(|\xi_n|^{-\text{Re} \,\zeta_j + \text{Re} \,\zeta_l - k-1}) = O(|\xi_n|^{-\delta_0-k}). 
\end{align}

We now consider the second summand, which is diagonal as it is the product of diagonal matrices. For any $\alpha \in \mathbb{C}$ and $\xi_n \rightarrow \pm \infty$, it will be useful to factorize $(\xi_n \pm i)^\alpha$ using the following identity:
\begin{equation*}
(\xi_n \pm i)^\alpha = (\xi_n \pm i0)^\alpha \, (1 \pm i \xi_n^{-1})^\alpha 
\end{equation*}
noting, as expected, that the decomposition on the right-hand side preserves the modulus and argument of the left-hand side. \\

For $\xi_n > 0$, the $(j,j)$ entry of the second summand is given by
\begin{align*}
(\xi_n-i)^{-\zeta_j} 1\, (\xi_n + i)^{\zeta_j} &= (\xi_n-i0)^{-\zeta_j} (\xi_n + i0)^{\zeta_j} 1 \,(1 - i \xi_n^{-1})^{-\zeta_j}(1 + i \xi_n^{-1})^{\zeta_j}\\
&= (1 - i \xi_n^{-1})^{-\zeta_j}(1 + i \xi_n^{-1})^{\zeta_j}
\end{align*}
since the product of the first two terms is $1$. \\

Similarly, for $\xi_n < 0$ we have
\begin{align*}
(\xi_n-i)^{-\zeta_j} \lambda_j (\xi_n + i)^{\zeta_j} &= (\xi_n-i0)^{-\zeta_j} (\xi_n + i0)^{\zeta_j} \lambda_j (1 - i \xi_n^{-1})^{-\zeta_j}(1 + i \xi_n^{-1})^{\zeta_j}\\
&= e^{-\zeta_j \log |\xi_n|} e^{i \zeta_j \pi} e^{\zeta_j \log |\xi_n|} e^{i \zeta_j \pi} \lambda_j (1 - i \xi_n^{-1})^{-\zeta_j}(1 + i \xi_n^{-1})^{\zeta_j}\\
&= e^{2 \pi i \zeta_j} \lambda_j (1 - i \xi_n^{-1})^{-\zeta_j}(1 + i \xi_n^{-1})^{\zeta_j}\\
&= (1 - i \xi_n^{-1})^{-\zeta_j}(1 + i \xi_n^{-1})^{\zeta_j}
\end{align*}
since $\lambda_j = e^{\log \lambda_j} = e^{- 2 \pi i \zeta_j}$ for $j = 1, \dots, N$ from equation \eqref{zetadef}. \\

So, for the second summand, combining the results for $\xi_n \rightarrow \pm \infty$  we have for $|\xi_n| > 1$
\begin{equation*}
(\xi_n - i)^{-\zeta} E_\pm (\xi_n + i)^{\zeta} - I  = (1 - i \xi_n^{-1})^{-\zeta}(1 + i \xi_n^{-1})^{\zeta} - I. 
\end{equation*}
So expanding the factors on the right-hand side in powers of $\xi_n^{-1}$  we have 
\begin{equation*}
(\xi_n - i)^{-\zeta} E_\pm (\xi_n + i)^{\zeta} - I = \sum^\infty_{l=1} A_l \xi_n^{-l} \quad \text{ for } |\xi_n| > 1.
\end{equation*}
Thus, on differentiating $k$ times with respect to $\xi_n$, we obtain
\begin{equation} \label{SecondtermD}
D^k_{\xi_n} \big \{(\xi_n - i)^{-\zeta} E_\pm (\xi_n + i)^{\zeta} - I \big \} = O ( |\xi_n|^{-1-k})
\end{equation}
for $k =0, 1, \dots, (r+3)$. Combining estimates \eqref{FirsttermD} and \eqref{SecondtermD}, we obtain
\begin{equation} \label{First+Second}
D^k_{\xi_n} \big \{ A^*_0(\xi^\prime, \xi_n) - I \big \} = O (|\xi_n|^{-\delta_0-k}), \quad |\xi_n| \rightarrow \infty.
\end{equation}
From Lemma \ref{Wrtest}, we have $A^*_0(\xi^\prime, \xi_n) \in W^{r+2}_{N \times N}(\mathbb {R})$. This completes the proof of the first step. \\
\end{proof}

With these preparations complete, we now turn to the general case. For convenience, we now restate Theorem \ref{MainResultCase3}.

\begin{lemma} 
Suppose that $A_0(\xi^\prime, \xi_n) \in C_{N \times N}^{r+3}(\mathbb{S}^{n-1})$ is a matrix-valued function which is homogeneous of degree $0$ and elliptic. Suppose that the Jordan form of $A_0^{-1}(0, \dots, 0,1) A_0(0, \dots, 0, -1)$ has blocks $J_k(\lambda_k)$ of size $m_k$ for $k=1, \dots, l$. Let $\zeta = (\zeta_1, \dots \zeta_N)$, where 
\begin{equation*}
\zeta_q = - (\log \lambda_j)/(2 \pi i)  \quad \text{ for } \sum^{j-1}_{p=1} m_p < q \leq \sum^{j}_{p=1} m_p, \quad q=1, \dots, N.
\end{equation*}
Let $c:= A^{-1}_0(0, \dots, 0,1)$. Then for fixed $\xi^\prime \not =0$,
\begin{equation}
A^*_0(\xi', \xi_n):= (\xi_n - i)^{-\zeta} B_-(\xi_n) h^{-1} c A_0(\xi^\prime, \xi_n) h B_+^{-1}(\xi_n) (\xi_n + i)^{\zeta} \in W^{r+2}_{N \times N}(\mathbb {R}),
\end{equation}
and
\begin{equation*} 
\lim_{\xi_n \to \pm \infty} A^*_0(\xi', \xi_n) = I. \\
\end{equation*}
\end{lemma}
\begin{proof} 
By hypothesis, and using equation \eqref{EB1}, we have 
\begin{equation*}
E= E_+^{-1}E_-  = h  \, \text{diag} \, \big [ \lambda_1 B^{m_1}(1), \dots, \lambda_l B^{m_l}(1) \big ] h^{-1}, 
\end{equation*}
for some invertible matrix $h$. If we define $\widetilde{A}_0(\xi^\prime, \xi_n) := h^{-1} c A_0(\xi^\prime, \xi_n) h$ we may assume, without loss of generality,  that 
\begin{equation*} 
E_+ = I; \quad E_-= \text{diag} \, \big [ \lambda_1 B^{m_1}(1), \dots, \lambda_l B^{m_l}(1) \big ]. \\
\end{equation*}
Mimicing the approach in the first case, we define
\begin{equation*}
\zeta^\prime_j = - (\log \lambda_j)/(2 \pi i); \quad -1/2 \leq \text{Re } \zeta^\prime_j < 1/2 \quad \text{where } j=1, \dots, l.
\end{equation*}
Moreover, we calculate
\begin{equation} 
\min_{1 \leq j,k \leq l} ( 1 - \text{Re } \zeta^\prime_k + \text{Re } \zeta^\prime_j) = \delta_0 > 0. \\
\end{equation}
We now define
\begin{equation*}
\zeta_q = \zeta^\prime_j  \quad \text{ for } \sum^{j-1}_{p=1} m_p < q \leq \sum^{j}_{p=1} m_p, \quad q=1, \dots, N.
\end{equation*}
Now we can set $\zeta = (\zeta_1, \dots \zeta_N)$, exactly as in the first case. \\

\begin{remark} 
When regarded as a function of $z \in \mathbb{C}$, the matrix-valued functions $B_\pm(z)$ are analytic in the regions Im $z > 0$ and Im $z < 0$ respectively. Note also that, by construction, the matrix-valued functions $B_\pm(\xi_n)$ commute with the diagonal matrix $(\xi_n \pm i )^\zeta$. [To see this, choose an arbitrary block $J_k$. On this block, $(\xi_n \pm i )^\zeta$ acts like a scalar, since the relevant components of the vector $\zeta$ are all equal to $- (\log \lambda_k)/(2 \pi i)$]. \\
\end{remark}

As previously, see equation \eqref{Astar2}, we define a new matrix-valued function
\begin{equation} \label{Astar3}
A^*_0(\xi^\prime, \xi_n) = (\xi_n - i)^{-\zeta} B_-(\xi_n) \widetilde{A}_0(\xi^\prime, \xi_n) B^{-1}_+(\xi_n)(\xi_n + i)^{\zeta}.
\end{equation}
Now, since $E_+ = I$, using the established properties of $B^m(\alpha_\pm)$ we have 
\begin{align*}
 \lim_{\xi_n \rightarrow \infty} & B_-(\xi_n) \widetilde{A}_0(\xi^\prime, \xi_n) B^{-1}_+(\xi_n) \\
& =  \lim_{\xi_n \rightarrow \infty} B_-(\xi_n) B^{-1}_+(\xi_n) \\
& =  \lim_{\xi_n \rightarrow \infty} B_-(\xi_n) \text{ diag }\big [ B^{m_1}(-\alpha_+(\xi_n)), \dots, B^{m_l}(-\alpha_+(\xi_n)) \big ]\\
& =  \lim_{\xi_n \rightarrow \infty} \text{diag }\big [ B^{m_1}(\alpha_-(\xi_n)  - \alpha_+(\xi_n)), \dots, B^{m_l}(\alpha_-(\xi_n) - \alpha_+(\xi_n)) \big ] \\
& = \text{diag }\big [ B^{m_1}(0), \dots, B^{m_l}(0) \big ] \\
& = I.
\end{align*}

On the other hand, since $E_-= \text{diag }\big [ \lambda_1 B^{m_1}(1), \dots, \lambda_l B^{m_l}(1) \big ]$,
\begin{align*}
\lim_{\xi_n \rightarrow -\infty} & B_-(\xi_n) \widetilde{A}_0(\xi^\prime, \xi_n) B^{-1}_+(\xi_n) \\
& =  \lim_{\xi_n \rightarrow -\infty} B_-(\xi_n) \text{ diag }\big [ \lambda_1 B^{m_1}(1), \dots, \lambda_l B^{m_l}(1) \big ]B^{-1}_+(\xi_n) \\
& =  \lim_{\xi_n \rightarrow -\infty} B_-(\xi_n) \text{ diag }\big [ \lambda_1 B^{m_1}(1-\alpha_+), \dots, \lambda_l B^{m_l}(1-\alpha_+) \big ]\\
& = \lim_{\xi_n \rightarrow -\infty} \text{diag }\big [ \lambda_{1}B^{m_1}(1-\alpha_+ + \alpha_-), \dots, \lambda_{l}B^{m_l}(1-\alpha_+ + \alpha_-) \big ]\\
& = \text{diag }\big [ \lambda_{1}B^{m_1}(0), \dots, \lambda_{l}B^{m_l}(0) \big ]\\
& = \text{diag }\big [ \lambda_{1}I^{m_1}, \dots, \lambda_{l}I^{m_l} \big ],
\end{align*}
where $I^m$ is an $m \times m$ block identity matrix. \\

So, as in Lemma \ref{Case2}, we see that
\begin{equation} \label{A*lim}
\lim_{\xi_n \rightarrow \pm \infty} A_0^*(\xi^\prime, \xi_n) = \lim_{\xi_n \rightarrow \pm \infty} (\xi_n - i)^{-\zeta} B_-(\xi_n) \widetilde{A}_0(\xi^\prime, \xi_n) B^{-1}_+(\xi_n)(\xi_n + i)^{\zeta} = I.
\end{equation}

To show that $D^k_{\xi_n} A_0^*(\xi^\prime, \xi_n), \, k = 1, \dots , (r+3)$ satisfies estimates of the form given in equation \eqref{First+Second}, we follow exactly the approach taken in Lemma \ref{Case2}. 
\begin{align*} 
A_0^*(\xi^\prime, \xi_n) = & (\xi_n - i)^{-\zeta}  B_-(\xi_n) \widetilde{A}_0(\xi^\prime, \xi_n) B^{-1}_+(\xi_n)(\xi_n + i)^{\zeta} \\
= & (\xi_n - i)^{-\zeta} B_-(\xi_n) (\widetilde{A}_0(\xi^\prime, \xi_n)- E_\pm)  B^{-1}_+ (\xi_n)(\xi_n + i)^{\zeta} \\
& + (\xi_n - i)^{-\zeta} B_-(\xi_n)  E_\pm  B^{-1}_+(\xi_n)(\xi_n + i)^{\zeta},
\end{align*}
where $B_-(\xi_n)  E_\pm  B^{-1}_+(\xi_n)$
\begin{equation*}
= E_\pm \, \text{diag} \bigg[ B^{m_j} \bigg( \log \dfrac{\xi_n - i}{\xi_n+i}\bigg) \bigg] = E_\pm \, \text{diag} \bigg[ B^{m_j} \bigg( \log \dfrac{1 - i / \xi_n }{1+ i/ \xi_n}\bigg) \bigg].
\end{equation*}

The presence of the logarithmic terms in the matrices $B_\pm$ adds only a minor complication. For any fixed positive integer $m$ we have
\begin{equation*}
D_{\xi_n} [\log (\xi_n \pm i)]^m = m [\log (\xi_n \pm i)]^{m-1} (\xi_n \pm i)^{-1}.
\end{equation*}
But since 
\begin{equation*}
\lim_{\xi_n \rightarrow \pm \infty} \dfrac{ [\log (\xi_n \pm i)]^p}{(\xi_n \pm i)^\epsilon} = 0,
\end{equation*}
for any fixed integer $p$ and any $\epsilon > 0$, we can effectively repeat the proof of Lemma \ref{Case2} with any $\delta^\prime$ satisfying $0 < \delta^\prime < \delta_0$. \\

From Lemma \ref{Wrtest}, we have $A^*_0(\xi^\prime, \xi_n) \in W^{r+2}_{N \times N}(\mathbb {R})$.  This completes the proof of the general case. \\
\end{proof}

\newpage
\section{Function approximation in $W^r(\mathbb{R})$} \label{AppWrApprox}
The goal in this appendix is prove that the Fourier transforms of smooth functions which have compact support and are zero in a neighbourhood of $x=0$, are dense in the space $W^r(\mathbb{R})$. To show this, we use the standard approach of cut-off functions and convolution with a mollifier. In simple terms, this analysis is required because we are effectively working in a \textit{weighted} Sobolev space. (See, for example, \cite{Ad}).\\

\begin{lemma}
Suppose $f \in W^r(\mathbb{R})$ and $0 \leq j \leq k \leq r$. Then 
\begin{equation*}
D^k f, \,\, t^j D^k f \text{  and  } D^k(t^j f) \in W(\mathbb{R}).
\end{equation*}
In addition, if $f = \widehat{g}$ then $\| x^k g \|_{L_1} = \| D^k f \|_W$, 
\begin{equation*}
\| D^j  (x^k g)  \|_{L_1} = \| t^j D^k f \|_W \quad \text{and} \quad \| x^k D^j g \|_{L_1} = \| D^k (t^j f) \|_W.
\end{equation*}
\end{lemma}
\begin{proof}
Since $W(\mathbb{R})$ is an R-algebra and $0 \leq j \leq k$,
\begin{equation} \label{rationallk}
\frac{t^j}{(1 - it)^k} \in W(\mathbb{R}).
\end{equation}

By definition, $(1 - it)^k D^k f \in W(\mathbb{R})$ and it follows from \eqref{rationallk} that
\begin{equation} \label{tlDkf}
t^j D^k f  = \frac{t^j}{(1 - it)^k}\, (1 - it)^k D^k f \in W(\mathbb{R}).
\end{equation}

Moreover, from \eqref{tlDkf}
\begin{equation*} 
D^k t^j f = \sum_{l = 0}^j c_l t^{j - l} D^{k - l} f \in W(\mathbb{R}),
\end{equation*}
where $c_l$ are some constants. (Note that $j \le k$ implies that $j-l \leq k-l$.) \\

From Proposition 2.2.11, p.\,100 \cite{Gr} 
\begin{equation*}
\mathcal{F}_{x \to t} \big [ (ix)^k g(x) \big ] = D^k \widehat{g}.
\end{equation*}
Hence, $\| x^k g \|_{L_1} = \| D^k \widehat{g} \|_W = \| D^k f \|_W$. \\

Let $h_k(x) := x^k g(x)$. Then, again from Proposition 2.2.11, p.\,100 \cite{Gr} 
\begin{equation*}
\mathcal{F}_{x \to t}  D^j h_k = (-it)^j \widehat{h_k},
\end{equation*}
and thus $\| D^j (x^k g) \|_{L_1} = \| t^j \widehat{x^k g} \|_{W}  = \| t^j D^k f \|_W$. \\

Finally, $\| x^k D^j g \|_{L_1} = \| \widehat{x^k D^j g} \|_W = \| D^k \widehat{D^j g} \|_W = \| D^k (t^j f )\|_W$.  \\

This completes the proof of the lemma. \\
\end{proof}

Suppose $\widehat{g}(t) \in W^{r}(\mathbb{R})$. Then, by definition,
\begin{equation*}
\| \widehat{g} \|_{W^r} = \| g \|_{L_1} + \sum^r_{k=1} \| (D_x +1)^k (x^k g(x)) \|_{L_1}.
\end{equation*}
Hence
\begin{equation*}
\| \widehat{g} \|_{W^r} \leq \| g \|_{L_1} + \sum^r_{k=1} \bigg \{ \|x^k g \|_{L_1}+ \sum^k_{j=1} \binom{k}{j} \| D^j_x (x^k g) \|_{L_1} \bigg \}.
\end{equation*}

Thus, to show convergence to $\widehat{g}$ in $\| \cdot \|_{W^r}$, it is sufficient to show convergence to $g, x^kg$ and $D^j_x(x^kg)$ in $\| \cdot \|_{L_1}$ for all $1 \leq j \leq k \leq r$. \\

Note that, for $j \geq 1$
\begin{equation*}
D^j_x (x^k g) = \sum^j_{l=0} \binom{j}{l} (D^l_x x^k) (D^{j-l}_x g).
\end{equation*}
Hence, we have
\begin{equation} \label{gstarnorm}
\| \widehat{g} \|_{W^r} \leq C_r \sum_{0 \leq j \leq k \leq r} \| x^k D^j_x  g \|_{L_1} := C_r \| g \|_*,
\end{equation}
where $C_r$ is a constant that \textit{only} depends on $r$. So, an alternative sufficient condition for the convergence to $\widehat{g}$ in $\| \cdot \|_{W^r}$ is the convergence to $x^k D^j_x g$ in $\| \cdot \|_{L_1}$ for all $0 \leq j \leq k \leq r$. \\

We now define a cut-off function that is zero in a neighbourhood of $x=0$, and is also equal to zero when $|x|$ is sufficiently large. Firstly, we define two smooth functions
\[
 \alpha(x) =
  \begin{cases}
   0 & \text{if } |x| \leq 1/2 \\
   1 & \text{if } |x| \geq 1,
  \end{cases}
\]
and 
\[
 \beta(x) =
  \begin{cases}
   1 & \text{if } |x| \leq 1\\
   0 & \text{if } |x| \geq 2.
  \end{cases}
\]
Then, for $0<  \epsilon < 1$, we define the smooth cut-off function $\phi_\epsilon$ by
\begin{equation*}
\phi_\epsilon(x) = \alpha(x/\epsilon) \beta (\epsilon x).
\end{equation*} 
Notice that, by construction, 
\[
\phi_\epsilon(x) = 
\begin{cases} 
1 & \mbox{if } |x| \in [\epsilon, \tfrac{1}{\epsilon}]  \\ 
0 & \mbox{if } |x| \in [0, \tfrac{\epsilon}{2})  \cup [\tfrac{2}{\epsilon}, \infty). 
\end{cases} 
\]
In particular, for each $0 < \epsilon <1$, the function $\phi_\epsilon$ has compact support, and is identically zero in a neighbourhood of $0$. \\

Therefore, for $j=1,2, \dots$, the support of  $D^j_x \phi_\epsilon(x)$ is contained in $E_\epsilon$, where
\begin{equation*}
E_\epsilon := [-\tfrac{2}{\epsilon}, -\tfrac{1}{\epsilon}] \cup [- \epsilon, -\tfrac{\epsilon}{2}] \cup[\tfrac{\epsilon}{2}, \epsilon] \cup [\tfrac{1}{\epsilon}, \tfrac{2}{\epsilon}].
\end{equation*} 

For any positive integer $k$, we have $D^k_x \alpha(x/\epsilon) = (1/\epsilon)^k  \alpha^{(k)}(x/\epsilon)$ and $D^k_x \beta(\epsilon x) = \epsilon^k \beta^{(k)}(\epsilon x)$. Hence, for $l=1,2, \dots$,
\begin{equation*} 
x^l D^l_x \phi_\epsilon  =  \sum^l_{k=0} c_k \bigg[  \bigg ( \dfrac{x}{\epsilon} \bigg)^k \alpha^{(k)} \bigg ( \dfrac{x}{\epsilon} \bigg) \bigg ] \bigg[  (\epsilon x)^{l-k} \beta^{(l-k)}(\epsilon x )\bigg] ,
\end{equation*}
for certain constants $c_k$ that only depend on $k$. Moreover, $\alpha^{(k)}(y) = 0$ unless $\tfrac{1}{2} \leq |y| \leq 1$ and $\beta^{(k)}(y) = 0$ unless $1 \leq |y| \leq 2$.\\

Hence, for $l=1,2, \dots$,
\begin{equation} \label{Cabl}
\sup_{x \in \mathbb{R}, \, 0< \epsilon<1} |x^l D^l_x \phi_\epsilon | \leq C_{\alpha, \beta, l},
\end{equation}
where $C_{\alpha, \beta, l}$ is a (finite) constant that depends \textit{only} on the smooth functions $ \alpha, \beta$ and the index $l$.\\

\begin{lemma} \label{L1gapprox}
Suppose $\widehat{g} \in W^{r}(\mathbb{R})$. Then $\| \widehat{\phi_\epsilon g} - \widehat{g} \|_{W^r} \to 0$ as $\epsilon \searrow 0$. \end{lemma}

\begin{proof}
Suppose $\widehat{g} \in W(\mathbb{R})$. Then, by definition, $ g(x) \in L_1(\mathbb{R})$ and $\| \widehat{g} \|_W = \| g \|_{L_1}$. It is immediately clear, from the definition of the $L_1$ norm, that  $ \phi_\epsilon g \in L_1(\mathbb{R})$ and $\| \phi_\epsilon g - g \|_{L_1} \to 0$ as $\epsilon \searrow 0$. That is, $\| \widehat{\phi_\epsilon g} - \widehat{g} \|_W \to 0$ as $\epsilon \searrow 0$, as required. \\

Now suppose $\widehat{g} \in W^{r}(\mathbb{R})$, and that $1 \leq k \leq r$. Then 
\begin{equation*}
\| x^k ( \phi_\epsilon g) - x^k g \|_{L_1} = \| \phi_\epsilon (x^k g) - x^k g \|_{L_1}  \to 0, \quad \text{ as } \epsilon \searrow 0. 
\end{equation*} 

Further suppose that $\widehat{g} \in W^{r}(\mathbb{R})$, and $1 \leq j \leq k \leq r$. Then
\begin{equation*}
D^j_x( \phi_\epsilon g) =  \phi_\epsilon (D^j_x g) + \sum^j_{l=1} \binom{j}{l} D^l_x \phi_\epsilon \cdot D^{j-l}_x g.
\end{equation*}
We now show that, for $1 \leq l \leq j \leq k \leq r$, 
\begin{align*}
\| x^k D^l_x \phi_\epsilon \cdot D^{j-l}_x g \|_{L_1} &= \int_{\mathbb{R}} | x^k D^l_x \phi_\epsilon \cdot D^{j-l}_x g | \, dx \\
& = \int_{E_\epsilon} | x^k D^l_x \phi_\epsilon \cdot D^{j-l}_x g | \, dx \quad \text{ since } D^l_x \phi_\epsilon = 0 \text{ outside } E_\epsilon \\
& \leq C_{\alpha, \beta, l} \int_{E_\epsilon} | x^{k-l} D^{j-l}_x g | \, dx \quad \text{ from } \eqref{Cabl}\\
& \to 0 \quad \text{ as } \epsilon \searrow 0,
\end{align*}
since $x^{k-l} D^{j-l}_x g \in L_1(\mathbb{R})$. Hence, $\| x^k D^j_x ( \phi_\epsilon g) - x^k D^j_x g \|_{L_1}$
\begin{align*}
 & \leq \| x^k \phi_\epsilon (D^j_x g) - x^k D^j_x g \|_{L_1}  + \sum^j_{l=1} \binom{j}{l} \| x^k D^l_x \phi_\epsilon \cdot D^{j-l}_x g \|_{L_1}\\
& = \|  \phi_\epsilon (x^k D^j_x g) - x^k D^j_x g \|_{L_1} + \sum^j_{l=1} \binom{j}{l} \| x^k D^l_x \phi_\epsilon \cdot D^{j-l}_x g \|_{L_1}\\ 
& \to 0, \text{ as } \epsilon \searrow 0.
\end{align*}
That is, $\| \widehat{\phi_\epsilon g} - \widehat{g} \|_{W^r} \to 0$ as $\epsilon \searrow 0$, as required. 
\end{proof}

\begin{remark}
The significance of Lemma \ref{L1gapprox} is that we can effectively assume, for the ensuing density arguments, that any function in $L_1(\mathbb{R})$ has both compact support and is also identically zero in a neighbourhood of the origin. (To see this, simply approximate $g \in L_1(\mathbb{R})$ by $h= \phi_\epsilon g$.) \\
\end{remark}

Following, for example \cite{Ad}, we now introduce the concept of a \textit{mollifier}. Let $J$ be a nonnegative, real-valued function in $C^\infty_0(\mathbb{R})$ satisfying the two conditions,  $J(x) = 0$ \text{ if } $|x| \geq 1$, and $\int_{\mathbb{R}} J(x) \, dx = 1$. \\

For $\delta > 0$, we define $J_\delta(x)  = \delta^{-1} J(x/ \delta)$. Then $J_\delta(x) \in C^\infty_0(\mathbb{R})$ and:
\begin{enumerate} [(a)]
{\setlength\itemindent{50pt}\item $J_\delta (x) = 0$ \text{ if } $|x| \geq \delta \text{ and }$} 
{\setlength\itemindent{50pt}\item $\int_{\mathbb{R}} J_\delta (x) \, dx = 1$.} 
\end{enumerate} 
As $\delta \searrow 0$, the mollifier $J_\delta(x)$ approaches the delta-function supported on $x=0$. Formally, we define the convolution
\begin{equation*}
(J_\delta * u)(x) = \int_{\mathbb{R}} J_\delta(x-y) u(y) \, dy.
\end{equation*}
Suppose $v \in L_1(\mathbb{R})$ and has compact support. Then:
\begin{enumerate} [(i)]
{\setlength\itemindent{50pt}\item $(J_\delta *v) \in C^\infty_0(\mathbb{R})$}
{\setlength\itemindent{50pt}\item $(J_\delta *v) \in L_1(\mathbb{R})$ and $\|J_\delta *v \|_{L_1} \leq \|v \|_{L_1}$} 
{\setlength\itemindent{50pt}\item \text{If } $Dv \in L_1(\mathbb{R})$, \text{ then } $D(J_\delta *v) = J_\delta *Dv$}
{\setlength\itemindent{50pt}\item $\lim_{\delta \searrow 0} \| J_\delta *v - v\|_{L_1} = 0$.}
\end{enumerate} 
As a simple consequence of the above, we observe that if $D^j v \in L_1(\mathbb{R})$ then
\begin{align*}
\| D^j (J_\delta *v) - D^j v\|_{L_1} & = \| (J_\delta *D^j v) - D^j v\|_{L_1} \\
& \to 0, \text{ as } \delta \searrow 0, 
\end{align*}
and, thus, $\| J_\delta *v -  v \|_{W^{r,1}} \to 0$ as $\delta \searrow 0$ in the (unweighted) Sobolev space $W^{r,1} (\mathbb{R})$.


\begin{lemma} \label{L1happrox}
Suppose $\widehat{h}(t) \in W^{r}(\mathbb{R})$, and further that $h(x)$ has compact support and is identically zero in a neighbourhood of $x=0$. Then $\| \widehat{J_\delta * h} -  \widehat{h} \|_{W^r} \to 0$ as $\delta \searrow 0$. 
\end{lemma}

\begin{proof}
Suppose that $h(x)$ has compact support and is identically zero in a neighbourhood of $x=0$. Then, there exist positive real numbers $\epsilon$ and $R$ such that supp $h \subseteq \{ x \in \mathbb{R} : \epsilon \leq |x| \leq R \} $. Suppose $\delta \leq \min  \{ \epsilon/2, 1 \}$. Then supp $J_\delta * h \subseteq \{ x \in \mathbb{R} : \epsilon/2 \leq |x| \leq (R+1) \} $. \\

Now let $H(x)$ be any function with supp $H \subseteq \{ x \in \mathbb{R} : \epsilon/2 \leq |x| \leq (R+1) \}$, and $\widehat{H}(t) \in W^{r}(\mathbb{R})$.  Then, for any integers $j,k$ such that $0 \leq j \leq k\leq r$, we have
\begin{equation*}
(\epsilon/2)^k \| D^j H \|_{L_1} \leq \| x^k D^j H \|_{L_1} \leq (R+1)^k \| D^j H \|_{L_1}.
\end{equation*}
Therefore
\begin{equation*}
(\epsilon/2)^r \| D^j H \|_{L_1} \leq \| x^k D^j H \|_{L_1} \leq (R+1)^r \| D^jH \|_{L_1},
\end{equation*}
and summing over all $0 \leq j \leq k\leq r$ we have
\begin{equation*}
(\epsilon/2)^r \| H \|_{W^{r,1}} \leq \| H \|_* \leq (r+1)(R+1)^r \| H \|_{W^{r,1}}, \\
\end{equation*}
where $\| \cdot \|_{W^{r,1}}$ denotes the (unweighted) Sobolev norm, and $\| \cdot \|_*$ is defined by equation \eqref{gstarnorm}. \\

Thus, since $\| J_\delta *h -  h \|_{W^{r,1}} \to 0$ as $\delta \searrow 0$, we have $\| J_\delta *h -  h \|_*$ as $\delta \searrow 0$. Finally, from equation \eqref{gstarnorm}, $\| \widehat{J_\delta * h} -  \widehat{h} \|_{W^r} \to 0$ as $\delta \searrow 0$, as required. \\
\end{proof}

\newpage
\section{Matrix factor estimates from Duduchava} \label{AppMatrixFactors}
In this appendix, we follow the approach taken by Duduchava \cite{Du}, and derive some asymptotic estimates for certain matrices arising during factorization. \\

Given $A^*_0(\xi', \xi_n) \in W^{r+2}_{N \times N}( \mathbb{R})$, we have the factorization
\begin{equation} \label{Astar0Factorization}
A^*_0(\omega, t) = (A^*_-(\omega,t))^{-1} \text{diag } \bigg ( \dfrac{t- i}{t+i}\bigg)^{\kappa(\omega)} A^*_+(\omega,t),
\end{equation}
where $A^*_\pm \in W^{r+2}_{N \times N}( \mathbb{R})$, and have analytic extensions, with respect to $\xi_n$, to the upper half-plane and lower half-plane respectively. Moreover, see p.\,37 \, \cite{GetAl}, since $\lim_{t \to \pm \infty} A^*_0(\omega,t) =I$, there exist factors $A^*_\pm \in W^{r+2}_{N \times N}( \mathbb{R})$ such that
\begin{equation} \label{limAstarpm}
\lim_{t \to \pm \infty} A^*_\pm (\omega,t) =I. \\
\end{equation}

We now define
\begin{equation} \label{ApmD1definition}
A^\pm_1(\omega,t) = (t \pm i )^\zeta A^*_\pm(\omega,t) (t \pm i)^{-\zeta}.
\end{equation}

\vspace{5mm}

We begin with two technical lemmas that will be useful later. Let $\mathbb{T}$ denote the unit circle in the complex plane. \\

\begin{lemma} \label{PhikHolder}
Let $0 < \nu < 1$. Suppose $\phi(t) \in W^{r+2}(\mathbb{R})$ and $\phi_k(t) := t^k D^k_t \phi(t) = O(|t|^{-\nu})$ as $|t| \to \infty$, for $k=0,1, \dots, r+2$. Define
\begin{equation*}
\Phi_k(z) := \phi_k \bigg( i\dfrac{1+z}{1-z} \bigg), \quad z \in \mathbb{T} \setminus \{ 1 \}; \quad \Phi_k(1) := \lim_{z \to 1} \Phi_k(z).
\end{equation*}
Then, for $j=0,1, \dots, r+1, \, \, \Phi_j \in H_\nu(\mathbb{T})$, where $H_\nu(\mathbb{T})$ denotes the H\"{o}lder space of order $\nu$. Moreover, $\Phi_j(1)=0$.\\
\end{lemma}
\begin{proof}
Choose any $z \in \mathbb{T}$. Then $z=e^{i \theta}$, for some $\theta \in [-\pi, \pi)$, and it is straightforward to show that
\begin{equation*}
i\dfrac{1+z}{1-z}  = - \cot \bigg ( \dfrac{\theta}{2} \bigg ).
\end{equation*}
Hence, for $j=0,1,\dots,r+1$, we can write 
\begin{equation*}
\Phi_j(z) = \phi_j(- \cot (\tfrac{\theta}{2})):= \psi_j (\theta).
\end{equation*}
In particular, as $\theta \to 0$ so $z \to 1$ and we obtain $\Phi_j(1)=0$.  Now

\begin{equation*}
\dfrac{d\psi_j}{d\theta} = \dfrac{d\phi_j}{d\tau} \, \dfrac{d\tau}{d\theta}, \quad \text{ where } \tau := - \cot (\tfrac{\theta}{2}).
\end{equation*}
By hypothesis,
\begin{equation*}
\dfrac{d\phi_j}{d\tau} = O(|\tau|^{-\nu-1}) = O(|\theta|^{\nu +1}), \quad \text{ as } |\theta| \to 0.
\end{equation*}
Moreover, by direct calculation
\begin{equation*}
\dfrac{d\tau}{d\theta} = \dfrac{1}{2 \sin^2 (\tfrac{\theta}{2}))} = O( |\theta|^{-2}), \quad \text{ as } |\theta| \to 0.
\end{equation*}
Combining these results,
\begin{equation} \label{DPhiTheta}
\dfrac{d\psi_j}{d\theta} = O( |\theta|^{\nu-1}), \quad \text{ as } |\theta| \to 0. 
\end{equation} \\

Now suppose $z_1, z_2 \in \mathbb{T}$. Then, by relabelling if necessary, we can suppose
\begin{equation*}
| z_1 - 1| \leq |z_2 -1|, 
\end{equation*} 
and we consider three cases:
\begin{itemize}[leftmargin=25mm]
\item [Case 1:] $|z_1 - z_2 | < | z_1 - 1| \leq |z_2 -1|$;
\item [Case 2:] $| z_1 - 1| \leq |z_1 - z_2 |  \leq |z_2 -1|$;
\item [Case 3:] $ | z_1 - 1| \leq |z_2 -1| < |z_1 - z_2 |$. \\
\end{itemize}

We begin with Case 1, and apply the Mean Value Theorem to $\psi_j$: 
\begin{align*}
|\Phi_j(z_1) - \Phi_j(z_2)| & = |\psi_j(\theta_1) - \psi(\theta_2)| \\
& =\bigg | \dfrac{d\psi_j}{d\theta}(\theta^*) \bigg | \cdot |\theta_1 - \theta_2| \quad (|\theta_1| \leq |\theta^*| \leq |\theta_2|),
\end{align*}
where, due to constraints applicable in this case, $\theta_1$ and $\theta_2$ must have the same sign. Hence, from \eqref{DPhiTheta},
\begin{align*}
|\Phi_j(z_1) - \Phi_j(z_2)| & \leq C' | z^*-1|^{\nu-1} \cdot |z_1- z_2| \quad \text{for some constant }C' \\
& = C' \bigg ( \dfrac{|z_1-z_2|}{|z^*-1|}\bigg )^{1-\nu} \, |z_1-z_2|^\nu \\
& \leq C' \bigg ( \dfrac{|z_1-z_2|}{|z_1-1|}\bigg )^{1-\nu} \, |z_1-z_2|^\nu \\
& \leq C'  \,  |z_1-z_2|^\nu.
\end{align*}

For Case 2,
\begin{align*}
|\Phi_j(z_1) - \Phi_j(z_2)| & \leq |\Phi_j(z_1)| +  |\Phi_j(z_2)| \\
& = |\phi_j(-\cot (\tfrac{\theta_1}{2})| +  |\phi_j(-\cot \tfrac{\theta_2}{2})| \\
& \leq C | \cot (\tfrac{\theta_1}{2})|^{-\nu} +  C | \cot (\tfrac{\theta_2}{2})|^{-\nu} \\
& \leq 2 \,C |\theta_1|^\nu + 2 \, C |\theta_2|^\nu \\
& \leq 4 \, C |z_1-1|^\nu + 4 \, C |z_2-1|^\nu
\end{align*}
But, in this case, $|z_1 -1 | \leq |z_1-z_2|$ and moreover,
\begin{equation*}
|z_2-1| \leq |z_2-z_1|+|z_1-1| \leq 2 |z_1-z_2|.
\end{equation*}
Therefore,
\begin{equation*}
|\Phi_j(z_1) - \Phi_j(z_2)| \leq 12 \, C |z_1-z_2|^\nu. 
\end{equation*} \\

Finally, turning to Case 3,
\begin{align*}
|\Phi_j(z_1) - \Phi_j(z_2)| & \leq |\Phi_j(z_1)| +  |\Phi_j(z_2)| \\
& = |\phi_j(-\cot (\tfrac{\theta_1}{2})| +  |\phi_j(-\cot \tfrac{\theta_2}{2})| \\
& \leq C | \cot (\tfrac{\theta_1}{2})|^{-\nu} +  C | \cot (\tfrac{\theta_2}{2})|^{-\nu} \\
& \leq 2 \, C |\theta_1|^\nu + 2 \, C |\theta_2|^\nu \\
& \leq 4 \, C |z_1-1|^\nu + 4 \, C |z_2-1|^\nu \\
& \leq 8 \, C |z_1-z_2|^\nu. 
\end{align*}
\end{proof}

\begin{lemma} \label{tkDkSR}
Let $0 < \nu < 1$ and $\phi(t) \in W^{r+2}(\mathbb{R})$. Suppose $\phi_k(t) := t^k D^k_t \phi(t) = O(|t|^{-\nu})$ as $|t| \to \infty$, for $k=0,1, \dots, r+2$. Then $ t^k D^k_t S_{\mathbb{R}} \phi(t) = O(|t|^{-\nu})$ as $|t| \to \infty$, for $k=0,1,\dots, r+1$.
\end{lemma}
\begin{proof}
By definition, for $k=0,1, \dots, r+2$,
\begin{align*}
t^k D^k_t S_\mathbb{R} \phi(t) & := \dfrac{t^k}{\pi i} D^k_t \int^\infty_{-\infty} \dfrac{\phi(\tau)}{\tau-t} \, d\tau\\
&= \dfrac{t^k}{\pi i}  \int^\infty_{-\infty} \dfrac{(D^k_\tau \phi)(\tau)}{\tau-t} \, d\tau \quad (\text{see}  \text{ Chapter I, Section 4.4 p.\,31} \, \cite{Ga}) \\
& = \dfrac{1}{\pi i} \int^\infty_{-\infty} \dfrac{\phi_k(\tau)}{\tau-t} \, d\tau, 
\end{align*}
where, in the last step, we use the identity
\begin{equation*}
t^k = \tau^k + (t - \tau) (t^{k-1} + t^{k-2} \tau + \dots + t \tau^{k-2} + \tau^{k-1}) ,
\end{equation*}
and repeated integration by parts.\\

For any $t \in \mathbb{R}$ we define the change of variable
\begin{equation*}
z:= \dfrac{t-i}{t+i} \quad \bigg (\text{or equivalently } t= i\dfrac{1+z}{1-z} \bigg ),
\end{equation*}
where, of course, $z \in \mathbb{T}$. Note that, as $t \to \pm \infty$ we have $z \to 1$. As previously, we define
\begin{equation*}
\Phi_k(z) := \phi_k \bigg( i\dfrac{1+z}{1-z} \bigg).
\end{equation*}
By Lemma \ref{PhikHolder},  $\Phi_k \in H_\nu(\mathbb{T})$ with $\Phi_k(1)=0$, for $k=0,1, \dots, r+1$.  \\

With this change of variable,
\begin{align*}
t^k D^k_t S_\mathbb{R} \phi(t) & = \dfrac{1}{\pi i} \int^\infty_{-\infty} \dfrac{\phi_k(\tau)}{\tau-t} \, d\tau \\
& =  \dfrac{1}{\pi i} \int_{|w|=1} \dfrac{1-z}{1-w} \dfrac{\Phi_k(w)}{w-z} \, dw \\
& =  \dfrac{1}{\pi i} \int_{|w|=1} \dfrac{\Phi_k(w)}{w-z} \, dw - \dfrac{1}{\pi i} \int_{|w|=1} \dfrac{\Phi_k(w)}{w-1} \, dw \\
& = (S_\mathbb{T} \Phi_k)(z) - (S_\mathbb{T} \Phi_k)(1).
\end{align*}
But the operator $S_\mathbb{T}$ is bounded on $H_\nu(\mathbb{T})$, and hence
\begin{equation*}
t^k D^k_t S_\mathbb{R} \phi(t) = O(|z-1|^\nu) = O(|t|^{-\nu}),
\end{equation*}
as $t \to \pm \infty \,\, (z \to 1)$. This completes the proof of the lemma. \\
\end{proof}

Our first task is to obtain some asymptotic estimates for the non-diagonal elements of $A^\pm_1$, Due to the similarity of the calculations, it is enough to prove this result for the matrix $A^+_1$. For brevity, we will ignore any constant terms that do affect the proof. \\

\begin{lemma} \label{lemmatAlog}
Suppose $1 \leq j,k \leq N$ with $j \not = k$. Then
\begin{equation*} 
D^q_t (A^\pm_1)_{j,k}(\omega,t) = O(|t|^{- \sigma -q }),
\end{equation*}
for $q=0,1, \dots, r+2$, and some $\sigma >0$. 
\end{lemma}
\begin{proof}
 We begin by noting that, from the definition of $A^+_1$,
\begin{equation*}
(A^+_1)_{j,k} = (t+i)^{\zeta_j-\zeta_k} (A^*_+)_{j,k}. 
\end{equation*} 

Firstly, we suppose that $\operatorname{Re} \, (\zeta_j - \zeta_k) < 0$. Then, in this case we can simply take 
\begin{equation*}
\sigma = - \operatorname{Re} \, (\zeta_j - \zeta_k),
\end{equation*}
so that $\sigma > 0$. Since $(A^*_+)_{j,k} \in W^{r+2}(\mathbb{R})$, the required result follows immediately.  We note that, taking $q=0$,
\begin{equation} \label{limtAlog}
\lim_{t \to \pm \infty} (A^+_1)_{j,k}(\omega,t)  = 0. 
\end{equation} 

Secondly, suppose that $j \not = k$ and $\operatorname{Re} \, (\zeta_j - \zeta_k) \geq 0$. From equation \eqref{Astar0Factorization}
\begin{align*}
A^*_+ - A^*_- & = \text{diag } \bigg ( \dfrac{t+ i}{t-i}\bigg)^{\kappa(\omega)} A^*_- A^*_0 - A^*_- \\
& = \bigg [  \text{diag } \bigg ( \dfrac{t+ i}{t-i}\bigg)^{\kappa(\omega)} - I \bigg] A^*_- + \text{diag } \bigg ( \dfrac{t+ i}{t-i}\bigg)^{\kappa(\omega)} (A^*_- A^*_0 - A^*_-) \\
& = \bigg [  \text{diag } \bigg ( \dfrac{t+ i}{t-i}\bigg)^{\kappa(\omega)} - I \bigg] A^*_- + \text{diag } \bigg ( \dfrac{t+ i}{t-i}\bigg)^{\kappa(\omega)} A^*_- (A^*_0 - I) \\
&:= b_1(t) + b_2(t) (A^*_0- I).
\end{align*}

Now consider $b_1(t) := \bigg [  \text{diag } \bigg ( \dfrac{t+ i}{t-i}\bigg)^{\kappa(\omega)} - I \bigg] A^*_-(\omega,t)$. We note that, as $t \to \pm \infty$,
\begin{align*}
D^q_t \bigg [  \text{diag } \bigg ( \dfrac{t+ i}{t-i}\bigg)^{\kappa(\omega)} - I \bigg] & = O(|t|^{-q -1}) \\
D^q_t [A^*_-(\omega,t)] &= O(|t|^{-q}) \quad \text{(since  } A^*_- \in W^{r+2}_{N \times N} (\mathbb{R})),
\end{align*}
for $q=0,1, \dots, r+2$. Hence, as $t \to \pm \infty$,
\begin{equation} \label{Dqtb1}
D^q_t b_1(t) = O( |t|^{-q-1}). \\
\end{equation} 

We now consider the second term,
\begin{equation*}
\big [b_2(A^*_0-I) \big ]_{j,k} = \sum^N_{s=1} (b_2)_{j,s}(t) (A^*_0(\omega,t) - I)_{s,k}(t)
\end{equation*}
where, by definition, $b_2(t) := \bigg ( \dfrac{t+ i}{t-i}\bigg)^{\kappa(\omega)} A^*_-(\omega,t)$. Since $A^*_- \in W^{r+2}_{N \times N}(\mathbb{R})$ we immediately have 
\begin{equation} \label{Dqtb2}
D^q_t (b_2)(t) = O(|t|^{-q}).
\end{equation}
Moreover, from estimates \eqref{FirsttermD} and \eqref{SecondtermD}
\begin{equation} \label{DqtA0I}
D^q_t (A^*_0 - I)_{s,k} = O(|t|^{-q - \text{Re} \, \zeta_s + \text{Re} \, \zeta_k + \epsilon - 1}).
\end{equation}
where $\epsilon$ is an arbitrarily small positive number that takes account of the logarithmic terms in the matrices $B_\pm(t)$ used in the construction of $A^*_0$. (See \eqref{A*lim}.)  Using estimates \eqref{Dqtb2} and \eqref{DqtA0I}
\begin{align} \label{Dqtb2A0}
D^q_t \big [b_2(A^*_0-I)] \big ]_{j,k} &= \sum^N_{s=1} O(|t|^{-q - \text{Re} \, \zeta_s + \text{Re} \, \zeta_k + \epsilon - 1}) \nonumber \\
&= \sum^N_{s=1} O(|t|^{\text{Re} \, (\zeta_k-\zeta_j) -q + \epsilon - \{ \text{Re} \, (\zeta_s -  \zeta_j)  + 1) \} }) \nonumber \\
&= O(|t|^{- \text{Re} \, (\zeta_j-\zeta_k) - q + \epsilon - \delta_0}).
\end{align}

Let $\nu = \text{Re} \, (\zeta_j - \zeta_k) + \delta_0 - \epsilon$. By assumption, $\text{Re} \, (\zeta_j - \zeta_k) \geq 0$ and hence we can choose any $\epsilon$ such that $0< \epsilon < \delta_0$ to ensure that $\nu >0$. Moreover, $\text{Re} \, (\zeta_j - \zeta_k) + \delta_0 \leq 1$, and hence $\nu < 1$.\\

Combining estimates \eqref{Dqtb1} and \eqref{Dqtb2A0}
\begin{equation*}
D^q_t (A^*_+ - A^*_-)_{j,k}(\omega,t) =O(|t|^{-\nu - q}),
\end{equation*}
for $q=0,1, \dots, r+2$, where $0< \nu < 1$. \\

For fixed $\omega$, we can now apply Lemma \ref{tkDkSR} with 
\begin{equation*}
\phi(t) = (A^*_+ - A^*_-)_{j,k}(\omega,t).
\end{equation*}
Let $\sigma := \delta_0 - \epsilon$. Then 
\begin{align*} 
D^q_t (A^*_+)_{j,k}(\omega,t) &= D^q_t \tfrac{1}{2}(I + S_\mathbb{R}) (A^*_+ - A^*_-)_{j,k}(\omega,t)  \\
& = O(|t|^{-\operatorname{Re} \, (\zeta_j - \zeta_k) - \sigma -q }).
\end{align*}
So, finally
\begin{equation} \label{DqtA+1}
D^q_t (A^+_1)_{j,k}(\omega,t) = O(|t|^{- \sigma -q }),
\end{equation}
for $q=0,1, \dots, r+1$, and $\sigma >0$. 
\end{proof}

\begin{lemma} \label{Apm1jnotk}
Suppose $1 \leq j,k \leq N$ with $j \not = k$. Then
\begin{equation*}
\lim_{t \to \pm \infty} (A^\pm_1(\omega,t))_{j,k} = 0.
\end{equation*}
\end{lemma}
\begin{proof}
The proof of the lemma follows directly from the estimates \eqref{limtAlog} and \eqref{DqtA+1}. (Of course, on using estimate \eqref{DqtA+1} we take $q=0$.) \\
\end{proof}

\begin{remark} \label{A1minusI}
From equation \eqref{limAstarpm}, $\lim_{t \to \pm \infty} (A^*_\pm)_{j,j} =1$, and hence
\begin{equation*}
\lim_{t \to \pm \infty} (A^\pm_1(\omega, t))_{j,j} = 1, \quad 1 \leq j \leq N.
\end{equation*}
Thus, the proof of Lemma \ref{lemmatAlog} can readily be extended to obtain (c.f. \eqref{First+Second})
\begin{equation*}
D^q_t (A^\pm_1(\omega,t) - I) = O(|t|^{-\sigma - q})
\end{equation*}
for $q=0,1, \dots, r+1$ and $\sigma > 0$. \\
\end{remark}

\begin{lemma} \label{Apm2inWr}
Suppose $1 \leq j,k \leq N$. Let
\begin{equation*}
A^\pm_2(\omega,t) = B^{-1}_\pm(t) (A^\pm_1(\omega,t) - I) B_\pm(t) + I 
\end{equation*} 
Then
\begin{equation*}
 (A^\pm_2(\omega,t))_{j,k} \in W^r(\mathbb{R}).
\end{equation*}
\end{lemma}
\begin{proof}
From Remark \ref{A1minusI} 
\begin{equation*}
D^q_t (A^\pm_1(\omega,t) - I) = O(|t|^{-\sigma - q}),
\end{equation*}
for $\sigma > 0$. Hence, from the definition of $A^\pm_2$,
\begin{equation*}
D^q_t (A^\pm_2(\omega,t) - I) = O(|t|^{-\sigma' - q}),
\end{equation*}
for $q=0,1, \dots, r+1$ and any $\sigma'$ such that $0< \sigma' < \sigma$. The required result now follows from Lemma \ref{Wrtest}.
\end{proof}

\end{document}